\documentclass[a4paper,12pt]{article}

\usepackage{amsmath}
\usepackage{amsthm}
\usepackage{amssymb}
\usepackage{cancel}
\usepackage{color}
\usepackage{calrsfs}
\usepackage{varioref}
\usepackage{bbm}
\definecolor{dblue}{rgb}{0.09,0.32,0.44} 
\usepackage[pdfborder={0 0 0}, colorlinks=true, pdfborderstyle={}, linkcolor=dblue, citecolor=dblue, urlcolor=blue]{hyperref}
\usepackage{epsfig}
\usepackage{psfrag}
\usepackage{subfigure}
\usepackage{graphicx}
\usepackage[mathscr,mathcal]{eucal}
\usepackage[shortlabels]{enumitem}
\usepackage{t1enc}
\usepackage[latin2]{inputenc}
\usepackage{appendix}

\newtheorem {theorem}{Theorem}
\newtheorem {lemma}{Lemma}
\newtheorem {corollary}{Corollary}
\newtheorem {proposition}{Proposition}

\newtheorem {conjecture}{Conjecture}
\newtheorem {remark}{Remark}

\def \B {\mathbb B}

\def \E {\mathbb E}

\def \K {\mathbb K}

\def \N {\mathbb N}

\def \P {\mathbb P}

\def \R {\mathbb R}
\def \T {\mathbb T}

\def \Z {\mathbb Z}

\def\cC{\mathcal{C}}

\def\cF{\mathcal{F}}
\def\cG{\mathcal{G}}

\def\cL{\mathcal{L}}

\def\vareps{\varepsilon}

\def\va{\mathbf{a}}

\def\vl{\mathbf{l}}

\def\vp{\mathbf{p}}

\def\vx{\mathbf{x}}
\def\vy{\mathbf{y}}

\newcommand{\probab}[1]{\ensuremath{\mathbf{P}\big(#1\big)}}
\newcommand{\expect}[1]{\ensuremath{\mathbf{E}\big(#1\big)}}

\newcommand{\condprobab}[2]{\ensuremath{\mathbf{P}\big(#1\bigm|#2\big)}}

\newcommand{\condexpect}[2]{\ensuremath{\mathbf{E}\big(#1\bigm|#2\big)}}
\newcommand{\condvar}[2]{\ensuremath{\mathbf{Var}\big(#1\bigm|#2\big)}}
\newcommand{\condcov}[3]{\ensuremath{\mathbf{Cov}\big(#1,#2\bigm|#3\big)}}

\newcommand{\probabz}[1]{\ensuremath{\mathbf{P}_{0}\left(#1\right)}}
\newcommand{\expectz}[1]{\ensuremath{\mathbf{E}_{0}\left(#1\right)}}

\def\clap#1{\hbox to 0pt{\hss#1\hss}}

\def\one{\ensuremath\mathbbm{1}}

\newcommand{\abs}[1]{\ensuremath\left|{#1}\right|}

\def \wt {\widetilde}

\def\wh{\widehat}

\newcommand{\lb}{\left(}
\newcommand{\rb}{\right)}
\newcommand{\lbr}{\left\{}
\newcommand{\rbr}{\right\}}

\newcommand{\be}[1]{\begin{equation}\label{#1}}
\newcommand{\ee}{\end{equation}}

\newcommand{\case}[1]{C{\small ASE}\,#1.}

\date{\today}

\title
{Split-and-Merge in Stationary Random Stirring on  Lattice 
Torus}

\author{
{\sc Dmitry Ioffe$^*$ and B\'alint T\'oth$^{\dagger,\$}$}
\\[8pt]
$^{*}$Technion -- Israel Institute of Technology, IL
\\
$^{\dagger}$University of Bristol, UK
\\
$^{\$}$R\'enyi Institute, Budapest, HU
}

\begin{document}

\maketitle

\begin{center}
{\em 
{
We dedicate this paper to Joel Lebowitz on the occasion of his 90th birthday 
\\
with deep respect for his scientific and moral accomplishment.
}
}
\end{center}

\begin{abstract}
{
We show that in any dimension $d\ge1$, the cycle-length process of stationary random stirring (or, random interchange) on the lattice torus converges to the canonical  Markovian \emph{split-and-merge} process with the invariant (and reversible) measure given by the Poisson-Dirichlet law $\mathsf{PD(1)}$, as the size of the system grows to infinity. In the case of transient dimensions, $d\ge 3$, the problem is motivated by attempts to understand the onset of long range order in quantum Heisenberg models via random loop representations of the latter.
}	
\end{abstract}

\section{Introduction and Result}

{
\subsection{General introduction}
\label{sub: General introduction}

{
Representations of { the} Bose gas in terms of random permutations date back to the classic \cite{feynman-53}, where the Feynman-Kac approach was first used in the context of quantum statistical physics. Since, due to Holstein-Primakoff transformations, quantum spin systems are reformulated as { the} lattice Bose gas with interactions, the Feynman-Kac approach can be transferred to the quantum Heisenberg models, too. An early version of representation of the spin-$\frac12$ quantum Heisenberg ferromagnet in terms of random permutations appears in the unjustly forgotten paper \cite{powers-76}. 

It looks like the stochastic permutation (or, random loop) approach to { the} Bose gas and quantum spin systems, based on Feynman-Kac, became main stream objects in mathematically rigorous quantum statistical physics and probability in the early nineties, with independent and essentially parallel works  where the Bose gas in continuum space \cite{suto-93}, the quantum Heisenberg ferromagnet on $\Z^d$ \cite{conlon-solovej-91, toth-91, toth-93}, and the quantum Heisenberg antiferromagnet in $\Z^1$ \cite{aizenman-nachtergaele-94}, had been considered, via random loop representations. 
{The latter paper contains a derivation of a general, 
	Poisson processes based,  functional integral  representations  of quantum spin states 
on finite graphs.}
We refer to \cite{ioffe-09} for a more recent  exposition of this general approach. \\

The random stirring (a.k.a. random interchange) process on a finite connected graph is a process of random permutations of its vertex-labels where elementary swaps are appended according to independent Poisson flows of rate one on unoriented edges. The process was first introduced by T. E. Harris, in \cite{harris-72} and since then, due to its manifold relevance and intrinsic beauty,  has been the object of abundantly many research papers. In particular, it turned out that the asymptotics of the cycle structure dynamics of random stirring on the $d$-dimensional discrete tori $\T^N$, as $N\to\infty$, is of paramount importance for understanding the emergence of so-called off-diagonal long range order in the spin-$\frac12$ isotropic quantum Heisenberg ferromagnet (for dimensions $d\ge3$) -- a Holy Grail of mathematically rigorous quantum statistical physics. For details, see \cite{toth-93} or the surveys \cite{ueltschi-11, ueltschi-13}. 

The main and best known conjecture in this context (see \cite{toth-93}) states that, for dimensions $d\ge3$, there exists a positive and finite critical time $\beta_c=\beta_c(d)$ beyond which cycles of macroscopic size of the random stirring emerge. 
For { a} precise formulation see Conjecture~\ref{con:toth} in section~\ref{sub:conjectures} below. 

{
Note that in the Feynman-Kac (a.k.a. imaginary time) setting the time parameter corresponds to inverse temperature. Accordingly, the critical value of time, $\beta_c$, corresponds,  in physical terms, to critical inverse temperature. This is reflected by our choice of notation.
}

Inspired by the exhaustive analysis of the Curie-Weiss mean field version of the problem by Schramm, cf \cite{schramm-05}, and supported by numerical evidence, a refinement of this conjecture (see \cite{ueltschi-11}) claims that beyond the critical time ${ \beta_c}$, the macroscopically scaled cycle lengths converge in distribution to the Poisson-Dirichlet law $\mathsf{PD(1)}$. For { a} precise formulation see Conjecture~\ref{con:ueltschi} in section~\ref{sub:conjectures} below.

The work presented in this note is primarily motivated by the following further refinement of the above conjectures. On the time scale of the random stirring process, due to the macroscopic number of edges connecting different cycles of macroscopic size, respectively, connecting different sites on the same cycle of macroscopic size,  the cycle structure of the permutation changes very fast. However, looking at a time-window of inverse macroscopic order  around  a fixed time ${\tau>\beta_c}$ and slowing down the time scale accordingly, we expect to see the cycles join and break up like in the canonical split-and-merge process. Somewhat refining Schramm's arguments,  \cite{schramm-05}, this can be proven in the Curie-Weiss mean field setup. In the $d$-dimensional setup, however, this seems to be a serious challenge, formulated as Conjecture~\ref{con:our} in section~\ref{sub:conjectures} below. The point is that in this scaling limit the underlying $d$-dimensional geometry is smeared out by the (expected) close-to-uniform spreading of the various macroscopic cycles. 

The main result of this note is formulated in Theorem~\ref{thm:coupl} and its Corollary~\ref{cor:weaklimit} in section~\ref{sub:result}, which settles Conjecture~\ref{con:our} for ${ \tau=\infty}$. That is, we prove that in the stationary regime of random stirring on $\T^N$, indeed, the appropriately rescaled and slowed down cycle-length process converges in distribution  to the canonical split-and-merge process, which has $\mathsf{PD(1)}$ as its unique stationary (and also reversible) law. 

}
\subsection{Notation}

Let $\Omega$ be the set of ordered partitions of $1$, 
\begin{align*}
\Omega
:= 
\left\{ 
\vp = (p_i)_{i\ge1}:
p_i\in[0,1], \ \ \  p_1 \geq p_2 \geq\dots\geq 0, \ \ \  \sum_{i} p_i =1
\right\}
\end{align*}
endowed with the $\ell^1$-metric 
\begin{align}
\label{ellone}
d(\vp, \vp^{\prime}):=\sum_i\abs{p_i-p^{\prime}_i}, 
\end{align}
which makes $\Omega$ a complete separable metric space. 

Given $N\in\N$, let $\Sigma^{N}$ be the symmetric group of all permutations of $\{1,\dots, N\}$ and 
\begin{align}
\notag
\Omega^{N}
:= 
&
\left\{ 
\vl=(l_i)_{i\ge1}: l_i\in\N, \ \ \ l_1\ge l_2\ge \dots \ge 0, \ \ \ \sum_{1}  l_i =N
\right\}
\\ 
\label{OmN}
= 
&
\left\{ 
\va=(a_k)_{k\ge1}: a_k\in\N, \ \ \ \sum_{k}  k a_k =N
\right\}.
\end{align}
The identification between the two representations of $\Omega^{N}$ is done through the formulas
\begin{align*}
a_k=\#\{i: l_i = k\}, 
\qquad
l_i=\max\{k: \sum_{k^\prime\geq k} a_{k^\prime}\geq i\}. 
\end{align*}
We embed naturally $\Omega^{N}\subset\Omega$ as 
\begin{align}
\label{OmNalt}
\Omega^{N}
= 
&
\left\{ 
\vp\in\Omega: p_i=\frac{l_i}{N}, \ \ \ l_i\in\N, \ \ \ l_1\ge l_2\ge \dots \ge 0, \ \ \ \sum_{1}  l_i =N 
\right\}, 
\end{align}
The three representations in \eqref{OmN} and \eqref{OmNalt} are naturally identified as three encodings of the same set $\Omega^{N}$. We will think about them as being the same and will use the three representations freely interchangeably. 

Given $\sigma\in\Sigma^{N}$ denote by $\cC(\sigma)=(\cC_i(\sigma))_{i\geq1}$ the cycle decomposition of the permutation $\sigma$, listed in decreasing  order of their sizes, so that in case of ties the order of cycles is given by the decreasing lexicographic order of their largest element. The cycle lengths of the permutation $\sigma\in\Sigma^{N}$ are encoded in the three (equivalent) maps: $\vl, \va, \vp: \Sigma^{N}\to\Omega^{N}$
\begin{align*}
l_i(\sigma)
:= 
\abs{\cC_i(\sigma)};
\qquad
a_i(\sigma)
:=
\#\{k:\abs{\cC_k(\sigma)}=i\};
\qquad
p_i(\sigma)
:=
\frac{\abs{\cC_i(\sigma)}}{N}. 
\end{align*}  
Let $\mu^{N}$ be the uniform distribution on $\Sigma^{N}$ and $\pi^{N}$ the probability distribution (on $\Omega^{N}$
) of the ordered cycle lengths of a uniformly sampled permutation from $\Sigma^{N}$: 
\begin{align*}
\pi^{N} (\vl)
:=
\mu^{N}(\sigma: \vl(\sigma)=\vl),
\quad
\pi^{N} (\va)
:=
\mu^{N}(\sigma: \va(\sigma)=\va),
\quad
\pi^{N} (\vp)
:=
\mu^{N}(\sigma: \vp(\sigma)=\vp).
\end{align*}
By Ewens's formula (see e.g.  \cite{arratia-barbour-tavare-03}) we have 
\begin{align}
\label{Ewens}
\pi^{N} (\va)
=
\left( \prod_{j}j^{a_j} a_j!\right)^{-1}, 
\end{align}
which transfers to $\pi^{N}(\vl)$ and $\pi^{N}(\vp)$ by the one-to-one identification of the three representations of $\Omega^{N}$ on the right hand side of \eqref{OmN}, \eqref{OmNalt}. Considering $\Omega^{N}$ as embedded in $\Omega$ (see \eqref{OmNalt}) the sequence of probability measures $\pi^{N}$ converges weakly to 

\noindent 
{\bf The Poisson-Dirichlet measure}  $\pi$ of parameter $\theta=1$ on $\Omega$. This is the distribution of the decreasingly ordered sequence $(\xi_j)_{j\geq 1}$, where 
\begin{align}
\label{xi-representation}
\xi_{j}:= 
\frac {\chi_j}{\sum_{k\geq1}\chi_k}, 
\qquad
(\chi_j)_{j\geq1}
 \sim 
 { \mathsf{PPP}}
 (m(dt)), 
\qquad
m(dt)= t^{-1}e^{-t}dt. 
\end{align}
Above 
{ $\mathsf{PPP}$} 
stands for  Poisson Point Process. See e.g.  Section~7 in \cite{ueltschi-11} for a concise exposition. 
{ We will also refer { to} the Poisson-Dirichlet law of parameter $\theta=1$, as $\mathsf{PD(1)}$.}

\subsection{Random Stirring on the $d$-dimensional torus}

The dimension $d$ will be fixed for ever in this note, and therefore it will not appear explicitly in notation. For $n\in \N$ and $N = n^d$ let  $\T^{N}:= (\Z/n)^d$ be the $d$-dimensional lattice torus of linear size $n$ and, accordingly, of volume $N$, and  $\B^{N}$ the set of nearest neighbour unoriented edges $\mathsf{b}$ of $\T^{N}$. We think about the vertices of the graph $\T^{N}$ as being listed in a fixed lexicographic order.

The random stirring (or, random transposition) process on $\T^{N}$ is the continuous time Markov process $t\mapsto\tilde \eta^{N}(t)$ on the { finite} state-space $\Sigma^{N}$, generated by independent Poisson flows (of rate one) of elementary transpositions $\tau_\mathsf{b}$ along the unoriented edges $\mathsf{b}\in\B^{N}$. Its infinitesimal generator, acting on functions $f:\Sigma^{N}\to\R$,  is 
\begin{align*}
\cL^{N}  
f (\sigma) 
= 
\sum_{\sf{b}\in \B^{N}} 
\left( f (\tau_\mathsf{b} \sigma) - f (\sigma)\right).
\end{align*}
The uniform distribution of permutations,  $\mu^{N}$, is the unique invariant measure of the Markov process $t\mapsto \eta^{N}(t)$ which is also reversible under this measure. 

In the sequel we shall work with appropriately rescaled (slowed down time) version $\eta^N$ of $\tilde\eta^N$, 
\begin{equation*}
\eta^N (t ) = \tilde\eta^N \lb \frac{t}{Nd}\rb . 
\end{equation*}
By construction  $\eta^N$ has unit total jump rates at any $\sigma\in \Sigma^N$.
We will consider the stationary process $t\mapsto\eta^{N}(t)$, with one-dimensional marginal distributions $\mu^{N}$.  
\smallskip 

\noindent
{\bf The process $\xi^N$.}
The main object of our note is the process of normalized and ordered cycle lengths of the 
{stationary} 
random stirring $\eta^{N}(t)$,  
\begin{align}
\label{xiN}
\xi^{N}(t):=\vp(\eta^{N}(t))
\end{align}
The process $t\mapsto \xi^{N}(t)$ takes values in $\Omega^{N}$ and it is stationary, with one dimensional marginals $\pi^{N}$, cf \eqref{Ewens}. However, it is by no means Markovian. As long as $N$ is finite, it reflects the geometry of the graph $\T^{N}$. Our result, Theorem \ref{thm:coupl} states, however, that, as $N\to\infty$, the process $\xi^{N}(t)$ stays close in distribution to a reversible Markovian coagulation-fragmentation process $t\mapsto \zeta^{N}(t)\in\Omega^{N}$ defined in the next subsection. Thus the process $\xi^{N}(t)$ inherits from its Markovian sibling $\zeta^{N}(t)$ the weak convergence to the canonical split-and-merge process $t\mapsto \zeta(t)\in\Omega$, also defined below. 

\subsection{Split-and-Merge}

The canonical split-and-merge process is a continuous time coagulation-fragmentation Markov process $t\mapsto \zeta(t)\in\Omega$ whose instantaneous jumps are either mergers of two different partition elements of size $p^{\prime}$ and $p^{\prime\prime}$ into one element of size $p^{\prime}+p^{\prime\prime}$ happening with rate $p^{\prime}p^{\prime\prime}$, or splitting of a partition element of size $p$ into two parts of sizes $p^{\prime}$ and $p^{\prime\prime}=p-p^{\prime}$, uniformly distributed in $[0,p]$, with rate $p^2$. Note, that the total rate of coagulation \emph{and} fragmentation events is exactly $1$. The { action of the} infinitesimal generator of the process { on bounded continuous}  $f:\Omega\to\R$,  is 
\begin{equation*}
\cG f (\vp) 
= 
2
\sum_{ i< j}p_i p_j 
\left( f (\mathsf{M}_{ij}\vp) - f (\vp )\right) 
+ 
\sum_i p_i^2 
\int_0^1 
\left( f (\mathsf{S}_i^u \vp ) - f (\vp )\right) du, 
\end{equation*}   
where, for $1 \leq i<j$,  the map $\mathsf{M}_{ij}:\Omega\to\Omega$ merges the partition elements $p_i$ and $p_j$ into one of size $p_i+p_j$, and  subsequently  rearranges the partition elements in decreasing order, whereas, for $1 \leq i$ and $u\in[0,1)$,  the map $\mathsf{S}_i^u:\Omega\to\Omega$ splits the partition element $p_i$ into two pieces of size $up_i$, respectively, $(1-u)p_i$  and  subsequently  rearranges the partition elements in decreasing order. 
{ Since the total rate of mergers and splittings is 
\begin{align*}
2
\sum_{ i< j}p_i p_j 
+ 
\sum_i p_i^2 
=1, 
\end{align*}
there is no technical issue with the path-wise construction of this process or with the identification of the domain of definition of its infinitesimal generator $\cG$.}
This canonical process is much studied and well understood. In particular, it is a known fact  -- see \cite{mayerwolf-zeitouni-zerner-02}, \cite{diaconis-mayerwolf-zeitouni-zerner-04} --   that the Poisson-Dirichlet measure $\pi$ on $\Omega$ is the unique stationary measure for the process $t\mapsto\zeta(t)$ which is also reversible under this measure. 
\smallskip

\noindent
{\bf The process $\zeta^N$}.
Given $N\in\N$, we define the { finite state space}  Markov process $t\mapsto\zeta^{N}(t)\in\Omega^{N}$ as a discrete (in space) approximation of $t\mapsto\zeta(t)\in\Omega$. It is the coagulation-fragmentation process of partition elements of size $k/N$, $k\in\{1,\dots,N\}$, where elements of size $k^{\prime}/N$ and $k^{\prime\prime}/N$ merge into an element of size  $(k^{\prime}+k^{\prime\prime})/N$ with rate $k^{\prime}k^{\prime\prime}/(N(N-1))$ and a partition element of size $k/N$ splits into two elements of sizes $k^{\prime}/N$ and $k^{\prime\prime}/N$, with
 { $k^{\prime} =1, \dots , k-1$ and}  $k^{\prime\prime}=k-k^{\prime}$, with rate $k/(N(N-1))$. Its infinitesimal generator, acting on functions $f:\Omega^{N}\to\R$, is 
\begin{equation}
\label{GN}
\begin{split}
\cG^{N} f (\vp)
=
&
\phantom{ + }
\frac{2 N}{N-1}  
\sum_{i< j} p_i p_j 
\left( f(\mathsf{M}_{ij}\vp) - f(\vp)\right) 
\\ 
&
+
{ 
\frac{1}{N-1}
\sum_{i} 
p_i
\sum_{k=1}^{Np_i-1} \left( f (\mathsf{S}_i^{k/(Np_i)} \vp) - f(\vp)\right)
}
\\ 
&{
:=  
\sum_{i< j} U^N_{{i,j}} (\vp)
\left( f(\mathsf{M}_{ij}\vp) - f(\vp)\right) + 
\sum_{j} 
\sum_{k=1} V_{j,k}^N (\vp )\left( f (\mathsf{S}_j^{k/(Np_j)} \vp) - f(\vp)\right)
}.
\end{split}
\end{equation}
 For future reference let us record the exact expressions for the jump rates above as
 \smallskip 
 
 \noindent
{\bf Mean-field (see Remark~\ref{rem:MF} below) jump rates.}
{ $U^{N}_{i,j}, V^{N}_{j,k}: \Omega^N\to [0,1]$, }
\be{eq:UVNrates} 
U^{N}_{i,j}(\vp )
:=
\frac{2N\one_{\{i<j\}}}{N-1} 
p_i p_j, 
\quad{\rm and}\quad
V^{N}_{j,k}(\vp )
:=
\frac{1}{N-1} p_j \one_{\{ 1\leq k< Np_j \}}.
\ee
Note that the total rate of mergers and splittings of the process $\zeta^{N}$  is also exactly 1. Indeed, 
\begin{align*}
\frac{2N}{N-1}\sum_{i<j}p_ip_j + \frac{1}{N-1}
{ \sum_i}\, 
p_i (Np_i-1) 
=1, 
\end{align*} 
which is just the combinatorial identity for the complete  probability of
	sampling two integers from $\lbr 1, \dots , N\rbr$ without replacement. 
{The process $\zeta^N$ with generator \eqref{GN}}  is also well understood and, in particular, it is known that Ewens's measure $\pi^{N}$ of \eqref{Ewens} is its (unique) stationary and reversible distribution
\cite{mayerwolf-zeitouni-zerner-02, diaconis-mayerwolf-zeitouni-zerner-04}.  

\begin{remark} 
	\label{rem:MF}
The process $t\mapsto \zeta^{N}(t)$ is actually the cycle length process of Curie-Weiss mean field random  stirrings. That is, 
{
\begin{align*}
\zeta^{N}(t)
= 
\vp\lb \wt\nu^{N}\lb \frac{2t}{N (N-1)}\rb \rb, 
\end{align*}
where $t\mapsto \wt\nu^{N}(t)$ is the 
{stationary} 
random stirring process on the complete graph $\K^{N}$ with unit stirring rate per unoriented edge.} However, this representation of the process $t\mapsto\zeta^{N}(t)$ will not be used later in this note.  
\end{remark} 

It is a well established fact -- see \cite{mayerwolf-zeitouni-zerner-02}, \cite{diaconis-mayerwolf-zeitouni-zerner-04} --  that, on any compact time interval $t\in[0,T]$,  the sequence of processes $t\mapsto\zeta^{N}(t)$  converges in distribution to the process  $t\mapsto\zeta(t)$, 
{as $N\to\infty$,}
in $\Omega$ endowed with the $\ell^1$-metric \eqref{ellone}.

\subsection{Result} 
\label{sub:result}

The results reported in this note are the following. 
 
\begin{theorem} 
\label{thm:coupl}
Let $d$ be fixed and $N=n^d$, $n \in \N$.  There exists a sequence $N\mapsto T^{*}(N)$ with $\lim_{N\to\infty}T^{*}(N)=\infty$ and a coupling (that is: joint realization on the same probability space) of the \emph{stationary} processes $t\mapsto \eta^{N}(t)$ and  $t\mapsto \zeta^{N}(t)$, with $\eta^{N}(0)\sim \mu^{N}$ and $\zeta^{N}(0)=\xi^{N}(0)$,  such that for any $\delta>0$
\begin{align}
\label{coupl}
\lim_{N\to\infty}
\probab{\max_{0\le t\le T^{*}(N)} d(\xi^{N}(t), \zeta^{N}(t))>\delta}
=0. 
\end{align}
\end{theorem}

\noindent
{\bf Note:}
In the coupling of Theorem \ref{thm:coupl} the marginal processes $t\mapsto \eta^{N}(t)$ and $t\mapsto \zeta^{N}(t)$ are stationary but the coupled pair $t\mapsto \left(\eta^{N}(t), \zeta^{N}(t)\right)$ is not. 

\begin{corollary}
\label{cor:weaklimit}
On any compact time interval $t\in[0,T]$
\begin{align*}
\xi^{N}(\cdot)
\Rightarrow
\zeta(\cdot),
\end{align*}
as $N\to\infty$, where $\Rightarrow$ denotes weak convergence in the space of c.a.d.l.a.g. trajectories in $\Omega$, endowed with the Skorohod topology based on the distance  \eqref{ellone}. 
\end{corollary}

\subsection{Conjectures}
\label{sub:conjectures}

In the following three conjectures 
the random stirring process $t\mapsto 
{\wt\eta}
^{N}(t)$ starts from the initial state $
{\wt\eta}
^{N}(0)=\mathsf{id}$ rather than being stationary
{and runs on the original time scale of unit stirring rate per edge.}  
We use subscript $0$ in  $\probabz{\cdot}$ to stipulate this initial condition. 

The conjectures are formulated in their 
increasing 
order of complexity: each being a natural refinement of the previous one. 

The basic and best known conjecture in the context of random stirrings on $\T^{N}$ is the "long cycle conjecture" originating in the stochastic representation of the spin-$\frac12$ quantum Heisenberg ferromagnet of T\'oth \cite{toth-93}. Affirmative settling of part (ii) of this conjecture would be essentially equivalent to proving existence of off-diagonal long range order at low temperatures for the isotropic spin-$\frac12$ quantum Heisenberg ferromagnet, in dimensions $d\geq3$ \  -- \ a Holy Grail of mathematically rigorous quantum statistical physics. For details see \cite{toth-93}. 

\begin{conjecture}
\label{con:toth}
(i) 
In dimension $d=2$, for any $t\in[0,\infty)$, any $i\in\{1,2,\dots\}$ and any $\varepsilon>0$
\begin{align}
\label{nogiant}
\lim_{N\to\infty}
\probabz{p_i(
{ \wt\eta}
^{N}(t))\ge \vareps}
=0.
\end{align}
(ii) 
In dimension $d\geq 3$ there exists ${ \beta_c}={ \beta_c}(d)\in(0,\infty)$, such that if $t\in[0,{ \beta_c})$ then \eqref{nogiant} holds, while if  $t\in({ \beta_c}, \infty)$ then  for any $i\in\{1,2,\dots\}$ and $\varepsilon$ sufficiently small
\begin{align*}
\varliminf_{N\to\infty}
\probabz{p_i(
{ \wt\eta}
^{N}(t))\ge \vareps}
>0.
\end{align*}
Furthermore, the function
\begin{equation*} 
m (t ) = \lim_{k\to\infty}\lim_{N\to\infty}
{
\sum_{i=1}^k 
}\, 
\expectz{p_i(
 \wt\eta^{N}(t))}
\end{equation*}
is a well defined  non-decreasing continuous function from  $[ 0,\infty)$ to $[0, 1]$, such that $m({ \beta_c})=0$,  $m(t)>0$ for $t>{ \beta_c}$, and $\lim_{t\to\infty}m(t)=1$. 
\end{conjecture}
{ 
The quantity $m(t)$ is the total fraction of sites belonging to cycles of macroscopic size, in the thermodynamical limit $N\to\infty$.
The probability that at least one transposition occurs across a bond ${\mathsf b}$ by time $t$ is $ = 1- {\rm e}^{-t}$, which may be viewed as a percolation probability across ${\mathsf b}$.   Evidently, at time $t$ macroscopic size permutation loops can lie only inside the corresponding 
 macroscopic connected clusters. 
In particular, $\beta_c (d)$ should be at least as large as 
$-\log \lb 1- q_c (d ) \rb$, where $q_c ( d)$ is the critical value for the Bernoulli nearest neighbour bond percolation on 
${\mathbb Z}^d$. 
Unlike in the Curie-Weiss mean field setting studied by Schramm \cite{schramm-05}, on $\Z^d$ the mass function $m(t)$ which appears in Conjecture~\ref{con:toth}  is 
strictly smaller 
than the density of the unique macroscopic-size 
percolation cluster, see the proof in the recent paper
\cite{muhlbacher-19}.
}
\smallskip 

Based on the mean-field (Curie-Weiss) results of Schramm \cite{schramm-05} and compelling numerical evidence Ueltschi et al. \cite{ueltschi-11, ueltschi-13} have formulated a refined version of  Conjecture~\ref{con:toth}, which not only affirms appearance of cycles of macroscopic size beyond a critical stirring time, but claims that the joint distribution of cycle lengths, rescaled by the total amount of gel,  weakly converges to the Poisson-Dirichlet measure $\pi$, just like in the mean field (Curie-Weiss) setting proved by Schramm \cite{schramm-05}. 
 
\begin{conjecture}
\label{con:ueltschi}
Assume $d\geq 3$ and let ${\beta_c}$ {and $m$}  be as in Conjecture \ref{con:toth} (ii)
and ${ \tau> \beta_c}$. 
For any $k\in \N$ and 
for any  bounded and continuous  function $f (\xi )= f (\xi_1, \dots , \xi_k)$, 
\begin{align*}
\lim_{N\to\infty}
\expectz{f(m({ \tau})^{-1}\vp(
\wt\eta^{N}({ \tau})))}
=
\int_{\Omega} fd\pi.
\end{align*}
\end{conjecture}

The work presented in this note is primarily motivated by the following further refinement of the above conjecture. { As indicated above the mass function $m (t )$ lives and grows on the time 
scale of the random stirring process $\wt\eta^{N}$. 
 On the other hand, 
the cycle structure of the permutation changes very fast
on this time scale}, due to the macroscopic number of edges connecting different cycles of macroscopic size, respectively, connecting different sites on the same cycle of macroscopic size. However, looking at a time-window of order $N^{-1}$ around  $\tau>{ \beta_c}$ and slowing down the time scale accordingly, we expect to see the cycles join and break up like in the canonical split-and-merge process $\zeta$. 

\begin{conjecture}
\label{con:our}
Under the conditions and notation of Conjecture \ref{con:ueltschi}, for $\tau> \beta_c$,
\begin{align*}
\left(t\mapsto m(\tau)^{-1}\vp(
{ \wt\eta}
^{N}(\tau+(Nd)^{-1}t))\right)
\Rightarrow
\left(t\mapsto \zeta(t)\right).
\end{align*}
\end{conjecture}

Corollary \ref{cor:weaklimit} is the special $\tau=\infty$ case of this conjecture. 

\subsection{Random loops in { the} quantum Heisenberg model} 
\label{sub:HM}
Let { ${\mathbf P}_{0}^{\beta}$ be the restriction of the random stirring measure ${\mathbf P}_{0}$ with initial condition $\mathsf{id}$ to the time interval 
	$[0, \beta ]$. 
Given a permutation  $\eta \in \Sigma^N$ let $\ell (\eta )$ denote the number of different cycles of $\eta$. 
}
 In the language of section~\ref{sub:conjectures} the 
isotropic spin-$\frac12$ Heisenberg { ferromagnet} at inverse temperature 
$\beta$ corresponds to a random stirring  $t\mapsto {\wt\eta}^{N}(t)$ 
on the time interval $[0, \beta ]$ subject to the  modified 
path measures
${\mathbf P}_{0}^{\theta , \beta}  \lb {\cdot}\rb $;  
\be{eq:Ptheta} 
{\mathbf P}_{0}^{\theta , \beta} 
\lb {\rm d}{\wt\eta}
^{N}\rb \propto 
\theta^{{ \ell} (\wt\eta^{N} (\beta ))} 
{ 
{\mathbf P}_{0}^\beta 
}
\lb {\rm d}{\wt\eta}
^{N}\rb , 
\ee
with $\theta =2$. Measures ${\mathbf P}_{0}^{\theta , \beta}$ 
with other values of $\theta\neq 2$ are perfectly well defined. 
{
As noted in \cite{ueltschi-13}, integer values $\theta=2,3,4,\dots$ are related to stochastic representations of quantum spin systems with spin $s=\frac{\theta-1}{2}$ 
{ with pair interactions}, 
which for $s=\frac12$ are exactly the isotropic ferromagnetic Heisenberg models, but for $s\geq 1$ are of more complex form. See \cite{ueltschi-13} for a fuller discussion.  (Fractional values of $\theta$ do not correspond to quantum spin systems.) 

On the other hand, as it was discovered and discussed in 
\cite{ueltschi-13},  in the 
{
$\theta=2$, or,
} 
spin-$\frac12$ case there is a whole family of modified stirring processes ${\mathbf P}_{0 ,{\mathsf u}}$ which interpolate between the ferromagnetic and anti-ferromagnetic models at the anisotropy parameter ${\mathsf u} \in [0,1]$. { In the  notation 
	of \cite{ueltschi-13} our  random stirring measure could be recorded as 
$	 
{\mathbf P}_{0}
= {\mathbf P}_{0 , {\mathsf 1}}
$. 
}
{This way \cite{ueltschi-13} provided an alloy of the random loop representations of the ferromagnetic (${\mathsf u} = 1$) and antiferromagnetic (${\mathsf u} = 0$) Heisenberg models, cf \cite{toth-93}, respectively, \cite{aizenman-nachtergaele-94}.}

In the { Curie-Weiss} mean field case, phase transition and Poisson-Dirichlet structure of ${\mathbf P}_{0 ,{\mathsf u}}${, for $\theta=1$ and ${\mathsf u}\in[0,1]$,} was worked out recently in \cite{BKLM18}, extending the study of the pure random stirring case, { $\theta=1$} and  ${\mathsf u}=1$, in  \cite{schramm-05}. { However}, even in the mean-field case { (Curie-Weiss)},   there are no direct  matching results for ${\mathbf P}_{0, {\mathsf u}}^{\theta , \beta}$ when $\theta \neq 1$.  The point is that for $\theta\neq 1$ the family of measures $\lbr {\mathbf P}_{0, {\mathsf u}}^{\theta , \beta}\rbr$ has  polymer structure: Namely,  ${\mathbf P}_{0,{\mathsf u}}^{\theta,\beta}$ is not a relativization of ${\mathbf P}_{0, {\mathsf u}}^{\theta , \beta^\prime}$ for $\beta <\beta^\prime$. In fact, under ${\mathbf P}_{0, {\mathsf u}}^{\theta , \beta}$ the process $\wt\eta^N$ is a continuous time Markov chain with time inhomogeneous jump rates $ J_{\eta , \eta^\prime}^{\theta , \beta }(t);\, t\in [0, \beta],$ given by 
\be{eq:j-rates} 
J_{\eta , \eta^\prime}^{\theta , \beta}(t) = \frac{h^{\theta , \beta}  (t, \eta^\prime)}{h^{\theta , \beta}  (t , \eta)} 
J_{\eta , \eta^\prime},\ \ {\rm where} \ \ h^{\theta , \beta}  (t , \eta) = 
{
{\mathbf E}_{0, {\mathsf u}} \lb \theta^{{ \ell} (\wt\eta^N (\beta ))} \, \big|\, \wt \eta (t ) = \eta 
\rb
}
, 
\ee
{ and  where $J_{\eta , \eta^\prime}$ are jump rates of  the modified stirring process ${\mathbf P}_{0, {\mathsf u}}$ (that is at $\theta = 1$)}.

In this respect, although Conjecture~1 is expected to hold as is,  it is not obvious what should be a proper reformulation of Conjectures~2-3 of the previous section for the family of measures $\lbr 
  {\mathbf P}_{0, {\mathsf u}}^{\theta , \beta}\rbr$. For instance, even if we assume Conjecture~1 and take $\beta > \beta_c$, { it is not clear what should be an adequate  answer to the following question: } Is it indeed  reasonable to expect that, for $t>\beta_c$  jump rates 
  $J_{\eta , \eta^\prime}^{\theta , \beta}(t)$ in  \eqref{eq:j-rates} are essentially constant  on 
  slowed down time scales of order $1/N$? 
  
  Furthermore, 
  it is not even clear what should be a proper formulation 
  of the stationary dynamics at $\beta = \infty$. As it was noted 
  in Section VIII of  \cite{ueltschi-13} the modified uniform measure   
  $\mu^{N, \theta} (\eta ) \propto  \theta^{{ \ell} (\eta )} \mu^N (\eta )$ is reversible with respect to the dynamics with jump rates 
  \be{eq:Jtheta} 
  J_{\eta ,\eta^\prime} \sqrt{\frac{\theta^{{ \ell} (\eta^\prime)}}{\theta^{{ \ell} (\eta )}}}  , 
  \ee
  but it is not clear whether jump rates \eqref{eq:Jtheta} could be recovered, as an appropriate limit, from 
  \eqref{eq:j-rates}. If, on the other hand, we take 
  \eqref{eq:Jtheta} as the definition of modified jump rates 
  for the random stirring on the lattice torus $\T^{N}$, 
  then, at least in the ${\mathsf u}=1$ case, there is 
  a straightforward adaptation of all the techniques and 
  ideas we develop below, which leads to a modification 
  of Theorem~\ref{thm:coupl} with  limiting  asymmetric split and merge 
  dynamics which is reversible with respect to the Poisson-Dirichlet law $\mathsf{PD(\theta )}$.  

\section{Proofs}

Our proof of Theorem~\ref{thm:coupl} is based on a coupling construction  which is developed in Subsection~\ref{sub:coupling}. This construction paves the way for a careful control of mismatch rates between processes $\zeta^N (t )$ and $\xi^N (t ) = {\mathbf p} (\eta^N (t ))$ which start at time zero at the same configuration sampled from Ewens's measure $\pi^N$ in \eqref{Ewens}; as developed in Subsections \ref{sub:mismatch}-\ref{sub:concluded}. 
{ We would like to stress  that the coupling which we construct here  \emph{is not} just a basic coupling where processes jump together with maximal possible rates. Splitting components of the dynamics happen to be too singular, and we need to introduce smoothing parameter $M$ as in \eqref{split-choice}.  $M\to\infty$ plays a crucial role in our main variance estimate \eqref{eq:var-bound} below. In this way we permit small 
alterations of the distance between the two processes
we try to couple, and control probabilities of big mismatches. This is, arguably, a novel idea, and we prefer to give full detail on the level of developing direct  upper bounds on probabilities of big mismatches.

In the concluding Subsection~\ref{sub:Gronwall} we sketch an alternative proof  via  Gr\"onwall's inequality, which gives an asymptotically  vanishing upper bound on $\expect{ \max_{s\leq t} d\lb \xi^N (s) , \zeta^N (s )\rb}$. This alternative proof is based on the very same coupling constructions and mismatch and variance estimates as developed in 
Subsections~\ref{sub:mismatch}-\ref{sub:comp} and a fully worked-out version would be of comparable length and complexity as the proof presented below.
} 

\subsection{Construction of coupling}
\label{sub:coupling}

All processes constructed below are piecewise constant and c.a.d.l.a.g. The ingredients of the construction are the following fully independent objects: 

\begin{itemize}

\item
The initial state $\eta^{N}(0)\sim\mu^{N}$ distributed uniformly on $\Sigma^{N}$. 

\item
A collection of i.i.d. Poisson processes of rate $(dN)^{-1}$, $\big(\nu_{\mathsf{b}}(t): \ \mathsf{b}\in\B^{N}\big)$. Their sum $\nu(t):= \sum_{\mathsf{b}\in\B^{N}} \nu_{\mathsf{b}}(t)$ is a Poisson process of rate $1$.  Denote $\theta_0=0$, $\theta_n$ the time of the $n$-th jump of the cumulative process $\nu(t)$ and by $\beta_n\in\B^{N}$ the edge on which the event occurred. 

\item 
Another Poisson process $\nu^{\prime}(t)$ of rate $1$. Denote $\theta^{\prime}_0=0$ and $\theta^{\prime}_n$ the time of the $n$-th jump of process $\nu^{\prime}(t)$. For later use let $\nu^{\prime\prime}(t):= \nu(t)+\nu^{\prime}(t)$ (a Poisson process of rate 2), and $(\theta^{\prime\prime}_n)_{n\geq0}$ the jump-times of this process (that is: the ordered sequence of $\{\theta_n : n\geq0\}\cup\{\theta^{\prime}_n: n\geq0\}$.)

\item
Two independent sequences of i.i.d. $\mathsf{UNI}([0,1])$ random variables,  $\alpha_n$, $\alpha^{\prime}_n$,  $n\geq1$ serving as source of extra randomness at the jump times $\theta_n$ and $\theta^\prime_n$, when needed.

\end{itemize}

\noindent
First we construct the slowed-down random stirring $t\mapsto \eta^{N}(t)$ as follows: 

\begin{enumerate}[-]

\item
$\eta^{N} (0)$ is sampled uniformly from $\Sigma^{N}$.

\item
$\eta^{N} (t)$ is constant in the intervals $[\theta_{n-1}, \theta_n)$, $n\geq1$. 

\item
At times $\theta_{n+1}$, $n\geq0$, $\eta^{N} (t)$ jumps from its actual value $\eta^{N}(\theta_{n})$ to $\eta^{N}(\theta_{n+1})=\tau_{\beta_{n+1}}\eta^{N}(\theta_{n})$.  

\end{enumerate}

\noindent 
Summarizing: $\eta^{N} (t)=\tau_{\beta_{n}}\dots \tau_{\beta_{1}} \eta^{N}(0)$ for $t\in[\theta_{n}, \theta_{n+1})$. As indicated in \eqref{xiN} we denote $\xi^{N}(t):=\vp(\eta^{N}(t))$.

In order to construct the process $t\mapsto \zeta^{N}(t)$ coupled to $t\mapsto \eta^{N}(t)$ we need some further notation. Let 
\begin{align*}
\cC^{N}_{i}(t)
:=
\cC_i(\eta^{N}(t)), 
\qquad
\qquad
\xi^{N}_{i}(t)
:=
\frac{\abs{\cC_i(\eta^{N}(t))}}{N}.
\end{align*}
For $1\leq i<j$ and an unordered pair of sites $\mathsf{b}$, let 
$\{ \cC_i \stackrel{\mathsf{b}}{\longleftrightarrow}\cC_j\}$ 
denote the event (in $\Sigma^{N}$) that the bond $\mathsf{b}{ \in \B^{N}}$ connects the cycles $\cC_i$ and $\cC_j$, and hence, under the transposition $\tau_\mathsf{b}$, they would merge into one cycle of length $\abs{\cC_i}+\abs{\cC_j}$. Similarly, For $1\leq i$, $1\leq k$  and an unordered pair of sited $\mathsf{b}{ \in \B^{N}}$, let 
$\{ \cC_i \stackrel{\mathsf{b, k}}{\longleftrightarrow}\cC_i\}$ 
denote the event that $1\leq k < \abs{\cC_i}$ and the bond $\mathsf{b}$ connects two elements of the cycle $\cC_i$ separated by exactly $k$-steps along the cycle. Note, that in this notation the events  
$\{ \cC_i \stackrel{\mathsf{b, k}}{\longleftrightarrow}\cC_i\}$ 
and 
$\{ \cC_i \stackrel{\mathsf{b, \abs{\cC_i}-k}}{\longleftrightarrow}\cC_i\}$ are the same. We introduce the indicators 
\begin{align}
\label{phi}
\varphi^{N}_{i,j,\mathsf{b}}(t)
&:= { \varphi^{N}_{i,j,\mathsf{b}}(\eta^N (t)):=}
\one_{\{i<j\}}
\one_{\{ 
\cC_{i}^N(t) 
\stackrel{\mathsf{b}}{\longleftrightarrow}
\cC_{j}^N(t)
\}},
\\
\label{psi}
\psi^{N}_{i,k,\mathsf{b}}(t) 
&:=
{ \psi^{N}_{i,k,\mathsf{b}}(\eta^N (t))
:=}
\big(\frac12 \one_{\{k\not=N\xi^N_i ( t )/2\}}+ \one_{\{k=N\xi^N_i ( t )/2\}} \big)
\one_{\{ 
\cC_{i}^N(t) 
\stackrel{\mathsf{b},k}{\longleftrightarrow}
\cC_{i}^N(t)
\}}. 
\end{align}
and the variables
\begin{align}
\label{XNd-var}
&
X^{N}_{i,j}(t) { := X^{N}_{i,j}(\eta^N (t))}
:=
\frac{1}{d N} 
\sum_{\mathsf{b} \in \B^{N}}
\varphi^{N}_{i,j,\mathsf{b}}(t),
\\
\label{YNd-var}
&
Y^{N}_{j,k}(t){ := Y^N_{j,k}(\eta^N (t))}
:=
\frac{1}{d N} 
\sum_{\mathsf{b} \in \B^{N}}
\psi^{N}_{j,k,\mathsf{b}}(t),
\\
\label{ZNd-var}
&
Z^{N}_{j,k}(t){ := Z^{N}_{j, k}(\eta^N (t))}
:=
\sum_{l}
w_{N\xi^N_j}(k,l)
Y^{N}_{j,l}(t)
=
\frac{1}{d N} 
\sum_{\mathsf{b} \in \B^{N}}
\sum_{l}
w_{N\xi^N_j}(k,l)
\psi^{N}_{j,l,\mathsf{b}}(t),
\end{align}
where the weights $w_{m}(k,l)$ are defined for $m\geq2$, $1\leq k,l\leq m-1$ and $M\in\N$ as follows:
\begin{align}
\label{w-def}
\begin{aligned}
&
\text{if } m<M+2:
&&
w_{m}(k,l):= 
\one_{\{1\leq k,l <m\}}
\frac{1}{m-1},
\\
&
\text{if } m\geq M+2:
&&
w_{m}(k,l):= 
\one_{\{1\leq k,l <m\}}
\times
\\
& 
&&
\hskip1cm
\begin{cases}
\displaystyle
1-\frac{\#\{l^\prime \in[1, m-1]: \abs{l^{\prime}-k}\in[1,  M]\}}{2M+1}
&\text{if } \abs{k-l}=0, 
\\
\displaystyle
\frac{1}{2M+1} 
&\text{if } \abs{k-l}\in[1,M], 
\\
\displaystyle 
0 
&\text{if } \abs{k-l}>M.
\end{cases}
\end{aligned}
\end{align}
Note that $w_{m}(k,l)=w_{m}(l,k)$ and $\sum_{k}w_{m}(k,l)=1$. 

The variables $X_{i,j}^N$ and $Y_{j,k}^N$ in 
	\eqref{XNd-var} and \eqref{YNd-var} describe instantaneous rates at which loops merge and split under the $\eta^N$-dynamics. 
	More precisely,  $X^{N}_{i,j}(t)$ is the instantaneous rate of merging 
{ $\cC_i^N(t)$} 
and 
{ $\cC_j^N(t)$}, 
and $Y^{N}_{i,k}(t)$ is the instantaneous rate of splitting 
{ $\cC_i^N(t)$} 
into two cycles of length $k$, respectively, 
{$\abs{\cC_i^N(t)}-k$}. 
Furthermore,  
\begin{equation*}
	\sum_{i, j}
	X^{N}_{i,j}(t)
	+
	\sum_{j}
	\sum_{k}
	Y^{N}_{j,k}(t)
	\equiv 1 
\end{equation*}	
The  proof of Theorem~\ref{thm:coupl} boils down to  verifying  that under the stationary  dynamics these rates are, 
in an  appropriate sense, close to the mean-field rates \eqref{eq:UVNrates}. Small cycles and exact splittings are harder to control.  Therefore, the  variables $Z^{N}_{j,k}$ represent 
cutoffs and randomization (or, in other words,  smoothening) of splitting rates $Y_{j,k}^N$ and they are designed 
in order  to facilitate the control of the $d$-distance in 
\eqref{coupl}. Note, however, that the total rate of splitting  is
preserved: For any cycle 
{ $\cC_j^N$}, 
\begin{align*}
\sum_{k}
Y^{N}_{j,k}(t)
\equiv 
\sum_{k}
Z^{N}_{j,k}(t)
\end{align*}

The parameter $M$ will be later chosen so that $1\ll M\ll N$, as $N\to\infty$.

Given the ingredients listed above, we construct the process $t\mapsto \zeta^{N}(t)$ as a piece-wise constant c.a.d.l.a.g. process on $\Omega^N$, as follows.

\begin{enumerate}[-]

\item
Start with $\zeta^{N}(0)=\xi^{N}(0)=\vp(\eta^{N}(0))$.

\item
Keep $\zeta^{N} (t)=\zeta^{N}(\theta^{\prime\prime}_{n})$  constant in the intervals $[\theta^{\prime\prime}_{n}, \theta^{\prime\prime}_{n+1})$, $n\geq0$. Recall  the
mean field rates \eqref{eq:UVNrates} and  let
\begin{align}
\label{UVN-var}
U^{N}_{i,j}(t)
:= U^{N}_{i,j}\lb \zeta^N (t )\rb 
\qquad\qquad
V^{N}_{j,k}(t)
:= V^{N}_{j, k}\lb \zeta^N (t )\rb .
\end{align}

\item
At times $\theta^{\prime\prime}_{n+1}$, $n\geq0$, $\zeta^{N} (t)$ jumps from its actual value $\zeta^{N}(\theta^{\prime\prime}_{n})$ as follows.

\begin{enumerate}[$\circ$]

\item
If $\theta^{\prime\prime}_{n+1}=\theta_m$ for some $m\geq1$ then

\begin{enumerate}[$\circ\circ$]

\item
If at time $\theta_m$, in the random stirring process $\eta^{N}$, the cycles $\cC^{N}_i$ and $\cC^{N}_j$ merge, then 
\begin{align}
\label{merge-choice}
\zeta^{N}(\theta^{\prime\prime}_{n+1})
=
\begin{cases}
\mathsf{M}_{ij} \zeta^{N}(\theta^{\prime\prime}_{n})
&
\mathrm{w. \ prob.}
\qquad
\displaystyle
\frac{\min \left\{X^{N}_{i,j}(\theta^{\prime\prime}_{n}), U^{N}_{i,j}(\theta^{\prime\prime}_{n})\right\}}
{X^{N}_{i,j}(\theta^{\prime\prime}_{n})},
\\[15pt]
\zeta^{N}(\theta^{\prime\prime}_{n})
&
\mathrm{w. \ prob.}
\qquad
\displaystyle
\frac{\left(X^{N}_{i,j}(\theta^{\prime\prime}_{n}) - U^{N}_{i,j}(\theta^{\prime\prime}_{n})\right)_+}
{X^{N}_{i,j}(\theta^{\prime\prime}_{n})},
\end{cases}
\end{align}

\item
If at time $\theta_m$, in the random stirring process $\eta^{N}$, the cycle $\cC^{N}_i$ splits into two cycles of lengths $k$, respectively, $\abs{\cC^{N}_i}-k$  then 
\begin{align}
\label{split-choice}
\zeta^{N}(\theta^{\prime\prime}_{n+1})
=
\begin{cases}
\mathsf{S}_{i}^{l/(N\zeta_i^{N}(\theta^{\prime\prime}_{n}))} 
\zeta^{N}(\theta^{\prime\prime}_{n})
&
\mathrm{w. \ prob.}
\\
&
\displaystyle
\frac{w_{N\xi^N_i(\theta^{\prime\prime}_{n})}(k,l)+w_{N\xi^N_i(\theta^{\prime\prime}_{n})}(N\xi^N_i(\theta^{\prime\prime}_{n})-k,l)}{2}
\times
\\
&
\displaystyle
\frac{\min\left\{Z^{N}_{i,l}(\theta^{\prime\prime}_{n}),V^{N}_{i,l}(\theta^{\prime\prime}_{n})\right\}}
{Z^{N}_{i,l}(\theta^{\prime\prime}_{n})},
\\[15pt]
\zeta^{N}(\theta^{\prime\prime}_{n})
&
\mathrm{w. \ prob.}
\\
&
\displaystyle
\sum_{l}
\frac{w_{N\xi^N_i(\theta^{\prime\prime}_{n})}(k,l)+w_{N\xi^N_i(\theta^{\prime\prime}_{n})}(N\xi^N_i(\theta^{\prime\prime}_{n})-k,l)}{2}
\times
\\
&
\displaystyle
\frac{\left(Z^{N}_{i,l}(\theta^{\prime\prime}_{n})-V^{N}_{i,l}(\theta^{\prime\prime}_{n})\right)_+}
{Z^{N}_{i,l}(\theta^{\prime\prime}_{n})}.
\end{cases}
\end{align}
Note, that the first alternative of \eqref{split-choice} makes sense only if $l<N \zeta_i^{N}(\theta^{\prime\prime}_{n})$. This, however, does not cause any formal problem in the above algorithm, as the probability of that alternative becomes $0$ if $l\geq N\zeta_i^{N}(\theta^{\prime\prime}_{n})$, see \eqref{UVN-var}. 

\end{enumerate}

Use the $\mathsf{UNI}([0,1])$-distributed random variable $\alpha_m$ to decide between the choices in \eqref{merge-choice}, respectively, \eqref{split-choice}.

\item
If $\theta^{\prime\prime}_{n+1}=\theta^{\prime}_m$ for some $m\geq1$ then
\begin{align}
\label{compensate-choice} \ \zeta^{N}(\theta^{\prime\prime}_{n+1})
=
\begin{cases}
\mathsf{M}_{ij} \zeta^{N}(\theta^{\prime\prime}_{n})
&
\mathrm{w. \ prob.}
\qquad
\displaystyle
\left(U^{N}_{i,j}(\theta^{\prime\prime}_{n})-X^{N}_{i,j}(\theta^{\prime\prime}_{n})\right)_+,
\\[15pt]
\mathsf{S}_{i}^{l/(N\zeta_i^{N}(\theta^{\prime\prime}_{n}))} \zeta^{N}(\theta^{\prime\prime}_{n})
&
\mathrm{w. \ prob.}
\qquad
\displaystyle
\left(V^{N}_{i,l}(\theta^{\prime\prime}_{n})-Z^{N}_{i,l}(\theta^{\prime\prime}_{n})\right)_+,
\\[15pt]
\zeta^{N}(\theta^{\prime\prime}_{n})
&
\mathrm{otherwise}.
\end{cases}
\end{align}
Use the $\mathsf{UNI}([0,1])$-distributed random variable $\alpha^{\prime}_m$ to decide between the choices in \eqref{compensate-choice}. 

\end{enumerate}

\end{enumerate}

From this construction it is clear that
\begin{enumerate}[$\circ$]

\item 
The jumps
$\zeta^N \to \mathsf{M}_{ij}\zeta^N$
occur with rate
\begin{align*}
\quad X^{N}_{i,j}(t)
\frac{\min \left\{X^{N}_{i,j}(t), U^{N}_{i,j}(t)\right\}}
{X^{N}_{i,j}(t)}
+
\left(U^{N}_{i,j}(t)-X^{N}_{i,j}(t)\right)_+
=
U^{N}_{i,j}(t) ,
\end{align*}

\item
The jumps
$\zeta^N \to \mathsf{S}_{i}^{l/(N\zeta_i^N)}\zeta^N$ 
occur with rate 
\begin{align*}
&
\sum_{k}
Y^{N}_{i,k}(t) 
\frac{w_{N\xi^N_i(\theta^{\prime\prime}_{n-1})}(k,l)+w_{N\xi^N_i(\theta^{\prime\prime}_{n-1})}(N\xi^N_i-k,l)}{2}
\cdot
\frac{\min \left\{Z^{N}_{i,l}(t), V^{N}_{i,l}(t)\right\}}
{Z^{N}_{i,l}(t)}
\\
&\qquad\qquad
+
\left(V^{N}_{i,l}(t)-Z^{N}_{i,l}(t)\right)_+
=
V^{N}_{i,l}(t), 
\end{align*}

\end{enumerate}

\noindent
not depending on the path $t\mapsto \eta^{N}(t)$,  and thus $t\mapsto \zeta^N(t)$ is exactly the Markovian split-and-merge process whose infinitesimal generator is $\cG^N$ given in \eqref{GN}.

\subsection{Mismatch rate}
\label{sub:mismatch} 
The process $t\mapsto (\eta^{N}(t), \zeta^N(t))$ constructed above is clearly a Markov jump process on the state space $\Sigma^N\times\Omega^N$. The jump at time $\theta^{\prime\prime}_{n+1}$ is called \emph{mismatched} if either the second case in \eqref{merge-choice} or \eqref{split-choice} occurs:
\begin{align*}
\theta^{\prime\prime}_{n+1}=\theta_m, 
\qquad
\eta^{N}(\theta^{\prime\prime}_{n+1})\not=\eta^{N}(\theta^{\prime\prime}_{n}),
\qquad 
\zeta^{N}(\theta^{\prime\prime}_{n+1})=\zeta^{N}(\theta^{\prime\prime}_{n}), 
\end{align*}
or the first or second case in \eqref{compensate-choice} occurs:
\begin{align*}
\theta^{\prime\prime}_{n+1}=\theta^{\prime}_m, 
\qquad
\eta^{N}(\theta^{\prime\prime}_{n+1})=\eta^{N}(\theta^{\prime\prime}_{n}),
\qquad
\zeta^{N}(\theta^{\prime\prime}_{n+1})\not=\zeta^{N}(\theta^{\prime\prime}_{n}).
\end{align*}
Denote by $\tau^N$ the time of first occurrence of a mismatched event: 
\begin{align*}
\tau^N
:=
\inf \{ t>0: 
&
(\eta^N(t^+)\not=\eta^N(t^-) \, \land \, \zeta^N(t^+)=\zeta^N(t^-))
\, \lor \, 
\\
&
(\eta^N(t^+)=\eta^N(t^-) \, \land \, \zeta^N(t^+)\not=\zeta^N(t^-))
\}. 
\end{align*}
A straightforward computation shows that the instantaneous rate of occurrence of $\tau^N$ is 
\begin{align}
\label{mismatch rate}
\varrho^{N}(t)
:=
\sum_{i,j} \abs{X^{N}_{i,j}(t)-U^{N}_{i,j}(t)} + 
\sum_{j,k} \abs{Z^{N}_{j,k}(t)-V^{N}_{j,k}(t)}. 
\end{align}
Recall the mean field rates \eqref{eq:UVNrates} and 
$\eta^N$-dependent flip rates \eqref{XNd-var}-\eqref{ZNd-var},  and  denote 
\begin{align}
\label{XNd-condfluct}
&
{\wh X}^{N}_{i,j}(t)
:= { U^N_{i,j} \lb \xi^N (t )\rb}= 
\frac{2N\one_{\{i<j\}}}{N-1} 
\xi^{N}_{i}(t)\xi^{N}_{j}(t) ,
&&
{\wt X}^{N}_{i,j}(t)
:=
{X}^{N}_{i,j}(t)
-
{\wh X}^{N}_{i,j}(t),
\\
\label{YNd-condfluct}
&
{\wh Z}^{N}_{j,k}(t)
:= { V^N_{j,k} \lb \xi^N (t )\rb} = 
\frac{1}{N-1} \xi^{N}_{j}(t) \one_{\{ 1\leq k< N \xi^{N}_{j}(t) \}}, 
&&
{\wt Z}^{N}_{j,k}(t)
:=
{Z}^{N}_{j,k}(t)
-
{\wh Z}^{N}_{j,k}(t).
\end{align}
As we shall see in Lemma~\ref{lem:cond-rates} below, the above quantities match proper centering
{of $X^N_{i,j}(t)$ and $Z^N_{j,k}(t)$, 
conditional on $\xi^N(t)$,
under the equilibrium uniform 
distribution of $\eta^N(t)$ on $\Sigma^N$.} 

{From \eqref{mismatch rate}, \eqref{XNd-condfluct} and \eqref{YNd-condfluct} we readily obtain the following upper bound on the mismatch rate $\varrho^N(t)$}
\begin{align}
\notag
\varrho^{N}(t)
&
\le 
\sum_{i,j} \abs{\wt X^{N}_{i,j}(t)} 
+ 
\sum_{j,k} \abs{\wt Z^{N}_{j,k}(t)}
+
\sum_{i,j} \abs{\wh X^{N}_{i,j}(t)-U^{N}_{i,j}(t)} 
+ 
\sum_{j,k} \abs{\wh Z^{N}_{j,k}(t)-\wh V^{N}_{j,k}(t)}
\\
\label{mismatch bound 1}
&
\le 
\sum_{i,j} \abs{\wt X^{N}_{i,j}(t)} 
+ 
\sum_{j,k} \abs{\wt Z^{N}_{j,k}(t)}
+
1{3} d(\xi^N(t), \zeta^N(t)). 
\end{align}
In the last step we have used the following straightforward estimates.
\begin{align}
\label{eq:UV-estimates}
\sum_{i< j} \abs{ U^{N}_{i,j}(\xi)-U^{N}_{i,j}(\zeta)} 
\le  
\frac{6N}{N-1} d(\xi, \zeta)
\ \text{and}\ \ 
\sum_{i,k} \abs{ V^{N}_{i,k}(\xi)-V^{N}_{i,k}(\zeta)} 
\le 
	\frac{6N}{N-1} d(\xi, \zeta).
\end{align}
The details of these 
{last}
computations are safely left for the reader. 
\smallskip

Next we bound from above the {$d(\xi^N(t), \zeta^N(t))$}-term on the right hand side of \eqref{mismatch bound 1}. The eventual bound is
	recorded in Corollary~\ref{cor:smalldistance} below. 
It is based on the following lemma, which is used to control
the growth of  $\ell^1$-distance under splits and merges:

\begin{lemma}
\label{lem:d-jumps}
For any $\vx, \vy \in\Omega$, $i,j\in\N$, 
{$i<j$,} 
and $u,v\in(0,1)$ the following hold:
{
\begin{align}
\label{merge-dist}
d(\mathsf{M}_{i,j}\vx,\mathsf{M}_{i,j}\vy)
&\le 
d(\vx,\vy)
\\
\label{split-dist}
d(\mathsf{S}^u_{i}\vx,\mathsf{S}^v_{i}\vy)
&\le 
d(\vx,\vy)
+
2\abs{ux_i-vy_i}
\\ 
\label{merge-mis}
d(\mathsf{M}_{i,j}\vx,\vy)
&\le 
d(\vx,\vy) 
+
x_j + y_j
\\
\label{split-mis} 
d(\mathsf{S}^u_{i}\vx, \vy)
&\le 
d(\vx,\vy)
+
\frac{x_i +y_i}{2} .
\end{align}
}
\end{lemma}
{
In the proof of Theorem \ref{thm:coupl} we will use only the bounds \eqref{merge-dist} and \eqref{split-dist}. The bounds \eqref{merge-mis} and \eqref{split-mis} will be used in the alternative sketch-proof of section \ref{sub:Gronwall}.}
 
In proving these bounds we rely on the alternative, equivalent expression of the $\ell^1$-distance on $\Omega$: 

\begin{lemma}
\label{lem:ellone-alt}
\begin{align}
\label{ellone-alt}
d(\vx,\vy) 
=
\inf_{\pi}\sum_{i}\abs{x_i-y_{\pi(i)}},
\end{align}
where the infimum is taken over all bijections of $\N$. 
{In view of \eqref{ellone} the infimum in \eqref{ellone-alt} is actually a minimum which is attained at the trivial bijection $\pi (i)\equiv i$.} 
\end{lemma}

\begin{proof}
[Proof of Lemma \ref{lem:ellone-alt}]
For any bijection $\pi:\N\to\N$, we have 
\begin{align*}
\sum_{j} \abs{x_j-y_{\pi(j)}}
= &{ \sum_j 
	 \int_0^\infty \lb \one_{x_j \leq t < y_{\pi (j )} }
		+  \one_{y_{\pi (j )} \leq t < x_j}\rb dt  }
	\\
& = \int_0^\infty 
\left(
{
\#\{j: x_j\leq t < y_{\pi(j)}\}
+
\#\{j:  y_{\pi(j)}\leq t < x_j\}
}
\right)
dt.
\end{align*}
However,
\begin{align*}
&
\#\{j: y_{\pi(j)}\leq t < x_j\}
\geq 
\left(
\#\{j: x_j>t\}
-
\#\{j: y_j>t\}
\right)_+
\\
&
\#\{j: x_j\leq t < y_{\pi(j)}\}
\geq 
\left(
\#\{j: x_j>t\}
-
\#\{j: y_j>t\}
\right)_-.
\end{align*}
Therefore, for any bijection $\pi:\N\to\N$, 
\begin{align*}
\sum_{j}\abs{x_j-y_{\pi(j)}}
\geq
\int_0^\infty 
\abs{
\#\{j: x_j>t\}
-
\#\{j: y_j>t\}}
dt
=
\sum_{j}\abs{x_j-y_j}, 
\end{align*}
where the last equality is just an expression for the excluded area  between two scaled  Young diagrams, 
and this completes the proof of the equality of the right hand sides of \eqref{ellone-alt} and \eqref{ellone}.
\end{proof}

\begin{proof}
[Proof of Lemma \ref{lem:d-jumps}] 
{
We will prove in turn the inequalities \eqref{merge-dist}, \eqref{split-dist}, \eqref{merge-mis} and \eqref{split-mis}. The proofs rely on Lemma~\ref{lem:ellone-alt} and elementary triangle inequalities.} 
\begin{align*}
\notag
d(\mathsf{M}_{i,j}\vx,\mathsf{M}_{i,j}\vy)
&
\leq 
\sum_{k:k\not=i,j}\abs{x_k-y_k} + \abs{x_i+x_j-y_i-y_j}
\\
\notag
&
\leq
\sum_{k:k\not=i,j}\abs{x_k-y_k} + \abs{x_i-y_i} + \abs{x_j-y_j} 
\\
&
=
d(\vx,\vy).
\end{align*}
\begin{align*}
\notag
d(\mathsf{S}^u_{i}\vx,\mathsf{S}^v_{i}\vy)
&
\leq 
\sum_{k:k\not=i}\abs{x_k-y_k} + \abs{ux_i-vy_i} + \abs{(1-u)x_i-(1-v)y_i}
\\
\notag
&
\leq
\sum_{k}\abs{x_k-y_k} + 2\abs{ux_i-vy_i} 
\\
&
=
d(\vx,\vy)+ 2\abs{ux_i-vy_i}.
\end{align*}
{
\begin{align*}
\notag
d(\mathsf{M}_{i,j}\vx,\vy)
&
\leq 
\sum_{k:k\not=i,j}\abs{x_k-y_k} + \abs{0-y_i} + \abs{x_i+x_j-y_j}
\\
\notag
&
\leq
\sum_{k:k\not=i}\abs{x_k-y_k} + x_i + y_i
\\
&
\leq
d(\vx,\vy) + x_i + y_i.
\end{align*}
\begin{align*}
\notag
d(\mathsf{S}^u_{i}\vx,\vy)
&
\leq 
\sum_{k:k\not=i}\abs{x_k-y_k} + \abs{\max\{u, 1-u\}x_i-y_i} + \abs{\min\{u,1-u\}x_i-0}
\\
\notag
&
\leq 
\sum_{k:k\not=i}\abs{x_k-y_k} + \max\{u, 1-u\}\abs{x_i-y_i} + \min\{u,1-u\}(x_i + y_i)
\\
&
\le
d(\vx,\vy)+ \frac{x_i + y_i}{2}.
\end{align*}
}
\end{proof}

\begin{corollary}
\label{cor:smalldistance}
As long as $t<\tau^N$, we have 
\begin{align}
\label{smalldistance}
d(\xi^N(t), \zeta^N(t))
\le 
\frac{2M}{N} \nu(t). 
\end{align}
\end{corollary}

\begin{proof}
[Proof of Corollary \ref{cor:smalldistance}]
Indeed, up to the first mismatch time $\tau^N$ the distance  $d(\xi^N(t), \zeta^N(t))$ will change only at the jump times $\theta_m$, when the first alternative in \eqref{merge-choice} or \eqref{split-choice} occurs. When a merge-event, \eqref{merge-choice} occurs, according to \eqref{merge-dist} the distance  $d(\xi^N, \zeta^N)$ does not increase.  On the other hand, when a split-event, \eqref{split-choice} occurs, then, according to \eqref{split-dist} the distance  $d(\xi^N, \zeta^N)$ increases by at most 
\begin{align*}
2 \abs{\frac{k}{N\xi^N_i}\xi^N_i - \frac{l}{N\zeta^N_i}\zeta^N_i}
=
\frac{2\abs{k-l}}{N}
\leq
\frac{2\abs{M}}{N}.
\end{align*} 
\end{proof}
From  \eqref{smalldistance} and  \eqref{mismatch bound 1} we get
\begin{align*}
\varrho^{N}(t)\one_{ \lbr t < \tau^N\rbr}
\le 
\sum_{i,j} \abs{\wt X^{N}_{i,j}(t)} 
+ 
\sum_{j,k} \abs{\wt Z^{N}_{j,k}(t)}
+
\frac{{26}M}{N} \nu(t), 
\end{align*} 
and hence
\begin{align*}
\notag
\condprobab{\tau^{N}<T}{\left(\eta^{N}(t)\right)_{t\ge0}}
&
=
1-
\exp\left\{-\int_0^T 
\left(
\sum_{i,j}
\abs{{\wt X}^{N}_{i,j}(t)}
+
\sum_{j,k}
\abs{{\wt Z}^{N}_{j,k}(t)} 
+
\frac{2{6}M}{N} \nu(t)
\right)dt 
\right\}
\\
&
\le 
\int_0^T 
\left(
\sum_{i,j}
\abs{{\wt X}^{N}_{i,j}(t)}
+
\sum_{j,k}
\abs{{\wt Z}^{N}_{j,k}(t)} 
+
\frac{{26}M}{N} \nu(t)
\right)dt. 
\end{align*}
Hence, exploiting stationarity of the process $t\mapsto \eta^{N}(t)$ we obtain 
\begin{align*}
\notag
\probab{\tau^{N}<T}
\le 
&
\phantom{ + }
T
{\mathbf E}\!
	\left(
\sum_{i,j}
\abs{{\wt X}^{N}_{i,j}}
\one_{\{\xi^{N}_j<\varepsilon\}}
+
\sum_{j,k}
\abs{{\wt Z}^{N}_{j,k}} 
\one_{\{\xi^{N}_j<\varepsilon\}}
\right)
\\
\notag
&
+ 
T
{\mathbf E}\!
\left(	
\sum_{i,j}
\abs{{\wt X}^{N}_{i,j}}
\one_{\{\xi^{N}_j\geq\varepsilon\}}
+
\sum_{j,k}
\abs{{\wt Z}^{N}_{j,k}} 
\one_{\{\xi^{N}_j\geq\varepsilon\}}
\right)
\\
&
+ 
1{3} T^2\frac{M}{N},
\end{align*}
where $\varepsilon>0$ is fixed for the moment and will be { sent} to $0$ at the end of the argument.

Next, using the straightforward upper bound
\begin{align*}
\sum_{i}
\abs{{\wt X}^{N}_{i,j}}
+
\sum_{k}
\abs{{\wt Z}^{N}_{j,k}} 
\le
{\frac{6 N}{N-1} } \xi^{N}_{j}, 
\end{align*}
which is direct consequence of the definitions of the variables 
$X^{N}_{i,j}$, $Z^{N}_{j,k}$, 
{
$\wh X^{N}_{i,j}$, $\wh Z^{N}_{j,k}$, $\wt X^{N}_{i,j}$, $\wt Z^{N}_{j,k}$ in  \eqref{XNd-var}, \eqref{ZNd-var}\eqref{XNd-condfluct} and, {respectively},    \eqref{YNd-condfluct}
}, we get
\begin{align}
\notag
\probab{\tau^{N}<T}
\le 
&
\phantom{ + }
1{3} T^2\frac{M}{N}
+
{7} T
\sum_{j}
\expect{
\xi^N_j
\one_{\{\xi^{N}_j<\varepsilon\}}
}
\\
\label{mismatch bound 4}
&
+ 
T
\sum_{i,j}
\expect{
\condexpect{\abs{{\wt X}^{N}_{i,j}}}{\xi^{N}}
\one_{\{\xi^{N}_j\geq\varepsilon\}}
}
+
\sum_{j,k}
\expect{
\condexpect{\abs{{\wt Z}^{N}_{j,k}}}{\xi^{N}} 
\one_{\{\xi^{N}_j\geq\varepsilon\}}
}.
\end{align}

\subsection{Variance estimates.}
\label{sub:comp} 

In order to simplify the formulas below we use generic notation $\mathbf{P}$ and $\mathbf{E}$  for the probability and expectation with respect to the  uniform measure $\mu^N$ on $\Sigma^N$.
This section is devoted to the proof of the following Lemma. 

\begin{lemma}
\label{lem:condexpect XY bounds}
\begin{align}
\label{condexpect X bound}
\condexpect{\abs{{\wt X}^{N}_{i,j}}}{\xi^{N}}
&
\leq
C N^{-1/2}
\one_{\{i<j\}} 
\sqrt{\xi^{N}_{i}\xi^{N}_{j}}
\\
\label{condexpect Z bound}
\sum_k
\condexpect{\abs{{\wt Z}^{N}_{j,k}}}{\xi^{N}}
&
\leq
{
C N^{-1/4}
{\xi^{N}_{j}}
}
\end{align}
\end{lemma}
Recall the variables $\varphi^{N}_{i,j,\mathsf{b}}$ and $\psi^{N}_{i,k,\mathsf{b}}$ from \eqref{phi} and \eqref{psi}. As we have already indicated,  variables 
${\wt X}^{N}_{i,j}$ and ${\wt Z}^{N}_{j,k}$ are centered under conditional expectation $\condexpect{\cdot}{\xi^{N}}
$. Here is the precise claim:

\begin{lemma} 
	\label{lem:condexp-phipsi} 
	\begin{align*}
	\condexpect{\varphi^{N}_{i,j,\mathsf{b}}}{\xi^N}
	&= { U_{i,j}^N (\xi^N)= }
	\frac{2N}{N-1} \one_{\{i<j\}}\xi^N_i \xi^N_j, 
	\\
	\condexpect{\psi^{N}_{i,l,\mathsf{b}}}{\xi^N}
	&= { V_{i,k} (\xi^N)= }
	\frac{1}{N-1} 
	\one_{\{ 1\leq l < N\xi^{N}_i\}}
	\xi^N_i .
	\end{align*}
	\end{lemma} 
\begin{proof}
[{Proof of Lemma~\ref{lem:condexp-phipsi}}]
The proof is a straightforward combinatorics for uniform
 distribution on $\Sigma^N$. 
\end{proof}

Consequently, the expressions on the left hand side of 
\eqref{condexpect X bound} and 
\eqref{condexpect Z bound} are bounded above as, 
\begin{align}
\label{condexpect X bound-var}
\condexpect{\abs{{\wt X}^{N}_{i,j}}}{\xi^{N}} 
&
\leq
\sqrt{\condvar{\frac{1}{dN}\sum_{\mathsf{b} \in \B^{N}}
		\varphi^{N}_{i,j,\mathsf{b}}}{\xi^{N}}}
	= \frac{1}{dN} \sqrt{\sum_{\mathsf{b,c} \in \B^{N}}
	\condcov{\varphi^{N}_{i,j,\mathsf{b}}}{\varphi^{N}_{i,j,\mathsf{c}}}{\xi^N}
}
\\
\notag
\condexpect{\abs{{\wt Z}^{N}_{j,k}}}{\xi^{N}}
&
\leq 
\sqrt{
	\condvar{\frac{1}{dN}\sum_{\mathsf{b} \in \B^{N}}
\sum_{l}
w_{N\xi^N_j}(k,l)
\psi^{N}_{j,l,\mathsf{b}}		}{\xi^{N}}}
\\ 
\label{condexpect Z bound-var}
&= \frac{1}{dN} 
\sqrt{ 
\sum_{\mathsf{b ,c} \in \B^{N}}	
\sum_{l,l^\prime} 
w_{N\xi^N_i}(k,l)w_{N\xi^N_i}(k,l^\prime)
\condcov{\psi^N_{i,l,\mathsf{b}}}{\psi^N_{i,l^\prime,\mathsf{c}}}{\xi^N}
}
\end{align}

In the following lemma we summarize the computational details on which the proof of Lemma~\ref{lem:condexpect XY bounds} relies. 
\begin{lemma}
	\label{lem:cond-bounds} 
	There exists a constant $C <\infty$ such that 
the following upper bounds hold uniformly in 
$N= n^d; n\in \mathbb{N}$: 
\begin{align} 
\label{eq:var-phi}
\abs{\condcov{\varphi^{N}_{i,j,\mathsf{b}}}{\varphi^{N}_{i,j,\mathsf{c}}}{\xi^N} } &\leq C\xi_i^N\xi_j^N\lb 
\one_{\abs{\mathsf{b}\cap\mathsf{c}} \leq 1} 
+ \frac{\xi^N_i +\xi_j^N}{N} \one_{\mathsf{b}\cap\mathsf{c}=\emptyset} 
\rb .\\
\label{eq:var-psi}
\abs{
\condcov{\psi^N_{i,l,\mathsf{b}}}{\psi^N_{i,l^\prime,\mathsf{c}}}{\xi^N}
}
&\leq 
C\frac{\xi^N_i}{N}\lb 
\one_{\mathsf{b}=\mathsf{c}}  \one_{l\sim l^\prime}
+ \frac{1}{N} \one_{\abs{\mathsf{b}\cap\mathsf{c}}=1} 
+ \frac{\xi^N_i}{N^2} \one_{\mathsf{b}\cap\mathsf{c}=\emptyset} 
\rb .
\end{align}
Above $l\sim l^\prime$ means that either $l= l^\prime$ or 
$l = N\xi^N_i -l^\prime$. 
\end{lemma}
\begin{proof}
[Proof of Lemma \ref{lem:cond-bounds}]
{The bounds \eqref{eq:var-phi} and \eqref{eq:var-psi}  follow directly from the exact formulas \eqref{cond-phi-squared} and \eqref{cond-psi-squared} stated in Lemma~\ref{lem:cond-rates} of the Appendix. }
\end{proof}

\begin{proof}
	[Proof of Lemma \ref{lem:condexpect XY bounds}]
	From
	\eqref{condexpect X bound-var} and 
	\eqref{eq:var-phi} 
	it follows that 
	\begin{align}
	\label{eq:tX-var}
	\condexpect{\abs{\wt  X^{N}_{i,j}}}{\xi^N}^2
	\leq  	\condvar{ X^{N}_{i,j}}{\xi^N} =
	\phantom{ + }
	\frac{1}{(Nd)^2}
	 \sum_{\mathsf{b},\mathsf{c}\in\B^{N}}
	\condcov{\varphi^N_{i,j,\mathsf{b}}}{\varphi^N_{i,j,\mathsf{c}}}{\xi^N}
	\leq
	\frac{C^\prime}{N}\xi^N_i\xi^N_j.
	\end{align}
In the last step we use \eqref{eq:var-phi}  in a straightforward way.  \eqref{condexpect X bound} follows. 
	
	Turning to \eqref{condexpect Z bound}, point-wise covariance estimates
	\eqref{eq:var-psi} imply that 
	\begin{align}
	\notag
	\condvar{ Z^{N}_{i,k}}{\xi^N}
	\leq
	&
	\phantom{ + }
	\frac{1}{(Nd)^2}
	\sum_{
		\mathsf{b}, \mathsf{c}\in\B^{N}}
	\sum_{l,l^\prime 
	}
	w_{N\xi^N_i}(k,l)w_{N\xi^N_i}(k,l^\prime)
	\condcov{\psi^N_{i,l,\mathsf{b}}}{\psi^N_{i,l^\prime,\mathsf{c}}}{\xi^N}
	\\
	\notag
	&
	\leq 
		\frac{1}{(Nd)^2}
	\sum_{\mathsf{b}\in\B^{N}}\frac{C\xi^N_i}{N}
	\sum_l \lb w_{N\xi^N_i}(k,l)^2 + 
	w_{N\xi^N_i}(k,l)w_{N\xi^N_i}(k,N\xi^N_i -l )\rb \\
	\notag 
	&+
	\frac{1}{(Nd)^2}
	\sum_{\stackrel {\mathsf{b}, \mathsf{c}\in\B^{N}}{\abs{\mathsf{b}\cap\mathsf{c}}=1}}
	\frac{C\xi^N_i}{N^2}
	\sum_{l,l^\prime
	}
	w_{N\xi^N_i}(k,l)w_{N\xi^N_i}(k,l^\prime )
	 \\ 
	&+
	\frac{1}{(Nd)^2}
	\sum_{\stackrel {\mathsf{b}, \mathsf{c}\in\B^{N}}{\mathsf{b}\cap \mathsf{c} =\emptyset}} 
	\frac{C (\xi^N_i )^2}{N^3}
	\sum_{l,l^\prime
	}
	w_{N\xi^N_i}(k,l)w_{N\xi^N_i}(k,l^\prime )
	\label{expect-ZND-squared}
	\end{align}
	The last two terms on the right hand side above are of order $1/N^3$. 
	From the definition \eqref{w-def} of the weights $w_m(k,l)$ it follows in  a straightforward way that 
	\begin{align*}
	\sum_{l} 
	\lb w_m(k,l)^2 + w_m (k , l )w_m (k, m-l)\rb 
	\leq
	\one_{\{\min\{k, m-k\}\leq M\}\}}
	+
	\frac{1}{2M+1}. 
	\end{align*}
	Plugging this into \eqref{expect-ZND-squared}, finally we get 
	\be{eq:var-bound}
	\condvar{ Z^{N}_{i,k}}{\xi^N}
	\leq
	C^\prime \xi^N_i 
	\left(
	\frac{1}{N^2}
	\one_{ \{ \min\{ k,  N\xi^N_i -k\} \leq M\} }
	+
	\frac{1}{N^2 M}
	+
	\frac{1}{N^3}
	\right)
	\one_{\{1\leq k< N\xi^N_i\}}
	\ee
	and hence, via Schwarz
	\begin{align*}
	\sum_{k=1}^{N\xi^N_i-1}
	\condexpect{\abs{\wt  Z^{N}_{i,k}}}{\xi^N}
	&\leq \sqrt{N\xi^N_i-1}\cdot \sqrt{\sum_k \condvar{ Z^{N}_{i,k}}{\xi^N}} \\
	&\leq 
	C^{\prime\prime} {\xi^N_i} 
	\sqrt{
	\frac{M}{N}
	+
	\frac{\xi^N_i}{{M}}
	+
	\frac{\xi^N_i}{{N}}
}.
	\end{align*}
	Finally, choosing $M=N^{1/2}$ we arrive at \eqref{condexpect Z bound}.
\end{proof}

\subsection{Proof of Theorem \ref{thm:coupl} -- concluded}
\label{sub:concluded}
Plugging ${M=N^{1/2}}$, \eqref{condexpect X bound} and \eqref{condexpect Z bound} into \eqref{mismatch bound 4} we obtain:
\begin{align}
\label{mismatch bound 5}
\probab{\tau^{N}<T}
\le 
&
C N^{-{1/2} }T^2
+
C  T
\sum_{j}
\expect{
	\xi^N_j
	\one_{\{\xi^{N}_j<\varepsilon\}}
}
+
\\
\notag
&
C N^{-1/2} T
\sum_{i<j}
\expect{\sqrt{\xi^{N}_i\xi^{N}_j}
	\one_{\{\xi^{N}_j\geq\varepsilon\}}
}
+
{
C N^{-1/4} T
\sum_{j}
\expect{{\xi^{N}_j}
	\one_{\{\xi^{N}_j\geq\varepsilon\}}
}.
}
\end{align}

\begin{lemma}
\label{lem:eps-to-zero-N-to-infty}
\begin{align}
\label{sum-xi-small-eps}
&
\lim_{\varepsilon\to0}
\lim_{N\to\infty}
\expect{
\sum_{j}
\xi^N_j
\one_{\{\xi^{N}_j<\varepsilon\}}
}
=
\lim_{\varepsilon\to0}
\expect{
\sum_{j}
\xi_j
\one_{\{\xi_j<\varepsilon\}}
}
=0,
\\
\label{sum-sqrt-xi-large-eps}
&
\lim_{\varepsilon\to0}
\lim_{N\to\infty}
\expect{
\sum_{j}
\sqrt{\xi^N_j}
\one_{\{\xi^{N}_j\geq\varepsilon\}}
}
=
\expect{
\sum_{j}
\sqrt{\xi_j}
}
<
\infty,
\\
\label{sum-sqrt-xi-squared-large-eps}
&
\lim_{\varepsilon\to0}
\lim_{N\to\infty}
\expect{
\sum_{i, j}
\sqrt{\xi^N_i\xi^N_j}
\one_{\{\max\{\xi^{N}_j,\xi^{N}_i\}\geq\varepsilon\}}
}
=
\expect{
\big(\sum_{j}
\sqrt{\xi_j}\big)^2
}
<
\infty.
\end{align}
\end{lemma}

\begin{proof}
[Proof of Lemma \ref{lem:eps-to-zero-N-to-infty}]
The $N\to\infty$ limits follow from uniform-in-$N$ boundedness: 
\begin{align*}
\sum_{j}
\xi^N_j
= 
1, 
\qquad
\sum_{j}
\sqrt{\xi^N_j}
\one_{\{\xi^{N}_j\geq\varepsilon\}}
\leq
\varepsilon^{-1/2}, 
\end{align*}
and dominated convergence. The $\vareps\to0$ limits follow from monotone convergence. 

It remains to prove the upper bound in \eqref{sum-sqrt-xi-squared-large-eps}. We will use the representation \eqref{xi-representation} of the joint distribution of the random variables $(\xi_i)_{i\geq1}$. Let $(\zeta_k)_{k\geq1}$ be the decreasingly ordered points of a Poisson point process on $\R_+$ with intensity $m(dt)=t^{-1}e^{-t}dt$. Then
\begin{align*}
\expect{\big(\sum_{k\geq1}\sqrt{\xi_k}\big)^2}
\leq
\expect{\big(\sum_{k\geq1}\sqrt{\zeta_k/\zeta_1}\big)^2}.
\end{align*}
However, the moment generating function of the random variable $\sum_{k\geq1}\sqrt{\zeta_k/\zeta_1}$ is explicitly computable, and finite for any $u\in\R$: 
\begin{align*}
\notag
\expect{\exp\{u \sum_{k\geq1}\sqrt{\zeta_k/\zeta_1} \}}
&
=
\int_0^\infty 
\condexpect{\exp\{u \sum_{k\geq1}\sqrt{\zeta_k/z}\} }{\zeta_1=z}
e^{-m([z,\infty))} dm(z).
\\
\notag
&
=
e^u
\int_0^\infty 
\exp\{\int_0^z (e^{u\sqrt{x/z}}-1) dm(x)\}
e^{-m([z,\infty))} dm(z)
\\
\notag
&
=
e^u
\int_0^\infty 
\exp\{\int_0^1 \frac{e^{u\sqrt{y}}-1}{y} e^{-zy}dy\}
e^{-m([z,\infty))} dm(z)
\\
&
\leq
e^u
\exp\{\int_0^1 \frac{e^{u\sqrt{y}}-1}{y}dy\}
<\infty. 
\end{align*}
\end{proof}

From \eqref{mismatch bound 5} and  \eqref{sum-xi-small-eps}, \eqref{sum-sqrt-xi-large-eps}, \eqref{sum-sqrt-xi-squared-large-eps} we conclude:

{
\begin{corollary}
\label{cor:mismatch time}
There exists a sequence $N\mapsto T^{*}(N)$ such that $\lim_{N\to\infty}T^{*}(N)=\infty$,  $\lim_{N\to\infty}N^{-1/4}T^{*}(N)=0$, and 
\begin{align*}
\lim_{N\to\infty}
\probab{\tau^N < T^{*}(N)}
=0.
\end{align*}
\end{corollary}

Finally, Theorem \ref{thm:coupl} follows from Corollaries \ref{cor:smalldistance} and \ref{cor:mismatch time} with $T^{*}(N)$ as in Corollary~\ref{cor:mismatch time}. 
}
	\subsection{A sketch of a direct approach using 
	Gr\"onwall's 
	inequality}  
	\label{sub:Gronwall} 
{	
	We continue to employ the simultaneous coupling 
	construction of the processes $\lb \eta^N (t ) , 
	 \zeta^N (t )\rb$ as introduced 
	in Subsection~\ref{sub:coupling}, and consider 
	$\xi^N (t ) = {\bf p}\lb \eta^N (t )\rb $.} 
	In particular, 
	$\eta^{N}(t)$ is stationary and reversible with respect to the uniform measure $\mu^{N}$ on $\Sigma^{N}$, and $\zeta^N(t)$ is stationary and reversible with respect to the
	Ewens's measure $\pi^N$ in \eqref{Ewens}. Furthermore, 
	at time zero  $\zeta^N (0) = { \xi}^N (0)$. In the sequel, 
	$\P$ and $\E$ denote the distribution and  the expectation of the process 
	$\lb \eta^N (\cdot  ) , \zeta^N (\cdot )\rb$, and 
	$\cF_t$ is the $\sigma$-algebra generated by 
	$\lbr \lb \eta^N (s   ) , \zeta^N (s )\rb\rbr_{s\in [0,t]}$. 
	
	Let us introduce the following notation: 
	\begin{align*} 
	\delta^N (t ) &=  \max_{s\leq t} d\lb \xi^N (s) , 
	\zeta^N (s )\rb\\ 
	{ \wt\delta^N} (t ) &= 
	\sum_{0\leq s\leq t} \Delta_+  
	d\lb \xi^N (s) , 
	\zeta^N (s )\rb , 
	\end{align*}
	where, for a piecewise constant cadlag function 
	$f$ we set  
\begin{equation*}	
\Delta_+ f (s ) = \max \lbr f (s ) - f (s^-) , 0\rbr .
\end{equation*}
	 Clearly , 
	\be{eq:delta-tdelta} 
	\delta^N (t ) \leq 
	{ \widetilde\delta^N} (t ) .
	\ee
	We claim that 
	\begin{proposition} 
		\label{prop:Gronwall} 
		There exists $C<\infty$ such that 
		\be{eq:Gronwall} 
		\E\lb  
		{ \widetilde\delta^N} (t )\rb \leq 
		C\lb \int_0^t \E \lb   
		{\widetilde\delta^N} (s )\rb
		{\rm d} s +\frac{t}{N^{1/4}}\rb 
		\ee 
		for all $N$ and $t$. 
		\end{proposition}
	By 
	Gr\"onwall's 
	inequality and \eqref{eq:delta-tdelta} we conclude: 
	\begin{corollary} 
		There exists $C<\infty$ such that 
		\be{eq:Bound}
		\E\lb \max_{s\leq t} d\lb \xi^N (s) , 
		\zeta^N (s )\rb\rb  \leq 
		\E\lb  
		\widetilde\delta^N (t )\rb \leq  
		\frac{C^\prime t}{N^{1/4}} , 
		\ee
		for all $N$ and $t$. 
		\end{corollary}
	Evidently, \eqref{eq:Bound} implies a somewhat quantitative version of Theorem~\ref{thm:coupl}. For the rest of this section we shall focus on sketching how 
	\eqref{eq:Gronwall} follows from the techniques and ideas
	developed in Subsections~\ref{sub:mismatch}-\ref{sub:comp}. 
	{We will, however, not spell out all details of the proof.}

	Recall our construction of coupling in Subsection~\ref{sub:coupling}. In particular recall that in the notation introduced therein jumps of either $\xi^N$ or $\zeta^N$ can occur only at arrival times 
	$\lb \theta_n^{\prime\prime}\rb$  of  Poisson proccess $\nu^{\prime\prime}$.   Let $t\in \lb \theta_n^{\prime\prime}\rb$, and 
	let us rely on Lemma~\ref{lem:d-jumps} for pinning down  possible expressions for 
	$\Delta_+ \widetilde{\delta}^N (t ) = \Delta_+ d \lb \xi^N (t ), \zeta^N (t )\rb$, and we shall use notation introduced in Subsection~\ref{sub:mismatch} for writing down expressions for instantaneous rates of occurence of the corresponding jumps. There are several
	cases to be recorded: \\
	\case{0} Neither $\eta^N$  nor $\zeta^N$  jumps. Then, 
	$\Delta_+ \widetilde{\delta}^N (t ) = 0$. \\ 
	\case{1} Matched merging of $\mathsf{M}_{ij}$ type. 
	In this case, 
	due to \eqref{merge-dist},
	$\Delta_+ \widetilde{\delta}^N (t ) = 0$. 
	\\
	\case{2}  
	Matched splittings of $\lb \mathsf{S}_i^u , \mathsf{S}_i^v\rb$
	type. In this case, 
	due to \eqref{split-dist}, 
 $\Delta_+ \widetilde{\delta}^N (t )\leq  2M/N$, which, due to our choice $M = \sqrt{N}$ below, 
	is just $2N^{-1/2}$.  The instantaneous rate of matched 
	 splittings of $\lb \mathsf{S}_i^u , \mathsf{S}_i^v\rb$
	type is at most $2$.  
\\ 
\case{3}  Mismatches of $\mathsf{M}_{ij}$ type. In view of \eqref{merge-mis}, in this case 
$\Delta_+ \widetilde{\delta}^N (t )\leq \xi^N_j (t) + \zeta^N_j (t )$. 
By construction the instantaneous rate of the $\mathsf{M}_{ij}$
mismatch is bounded above by 
\begin{equation*} 
 \abs{X_{i,j}^N(t ) - U_{i,j}(\zeta^N (t))} \leq  
 \abs{X_{i,j}^N(t ) - 
 	\condexpect{X_{i,j}^N (t )}{\xi^N (t )}} 
 + \abs{U_{i,j}^N(\xi^N (t))
 	- U_{i,j}^N(\zeta^N (t))}  .
\end{equation*}
\case{4} Mismatches of $\mathsf{S}_i^u$ type. By \eqref{split-mis}, in this case $\Delta_+ \widetilde{\delta}^N (t )\leq \lb  \xi^N_i (t) + \zeta^N_i (t )\rb /2$. For $u \in 
\lbr k/( N\xi^N_i (t )) , k/( N\zeta^N_i (t ))\rbr $ such mismatches 
occur at  instantaneous rates  
\begin{equation*} 
\abs{Z_{j,k}^N (t ) -V^{N}_{j,k}\lb \zeta^N (t)\rb} \leq 
\abs{Z_{j,k}(t ) - 
	\condexpect{Z_{j ,k}^N(t )}{\xi^N (t )}} 
+ \abs{Z_{j, k}^N(\xi^N (t))
	- Z_{j ,k}^N(\zeta^N (t))}  .
\end{equation*}
We conclude: 
\begin{lemma} 
\label{lem:d-growth} 
The following upper bound on instantaneous growth 
of $\widetilde{\delta}^N (t )$ holds: 
\begin{align} 
\notag
&\limsup_{h\to 0} \frac{1}{h}\E\lb 
\widetilde{\delta}^N (t + h ) - \widetilde{\delta}^N (t )~\big| ~ \cF_t\rb 
 \leq \frac{4M}{N} \\ 
 \label{Case3}
&\quad + 
\sum_{i <j} \lb \abs{X_{i,j}^N(t ) - 
	\condexpect{X_{i,j}^N (t )}{\xi^N (t )}} 
+ \abs{U_{i,j}^N(\xi^N (t))
	- U_{i,j}^N(\zeta^N (t))} \rb
\lb \xi^N_j (t ) + \zeta^N_j (t )\rb\\
\label{Case4}
&\quad 
+ 
\sum_{j , k} \lb \abs{Z_{j,k}^N (t ) - 
	\condexpect{Z_{j , k}^N(t )}{\xi^N (t )}} 
+ \abs{Z_{j , k}(\xi^N (t))
	- Z_{j, k}(\zeta^N (t))} \rb
\frac{\xi^N_j (t ) +\zeta_j^N (t )}{2} .
\end{align} 
\end{lemma} 
Indeed, the three terms  on the right hand side above  correspond to \case{2}-\case{4} just discussed. 
 Let us derive upper bounds on the $\E$-expectations of the sums \eqref{Case3} and \eqref{Case4}. 

\noindent
{\bf Upper bound on the $\E$-expectation of \eqref{Case3}.} 
Let us start with the second term in \eqref{Case3}. By the first of \eqref{eq:UV-estimates}
\begin{equation*} 
\sum_{i <j} 
\abs{U_{i,j}^N(\xi )
	- U_{i,j}^N(\zeta ) }
\lb \xi_j  + \zeta_j \rb \leq 
2\sum_{i <j} 
\abs{U_{i,j}^N(\xi )
	- U_{i,j}^N(\zeta ) } \leq 
\frac{12N}{N-1} d (\xi , \zeta ). 
\end{equation*}
Next, as far as the expectation of the first summand in \eqref{Case3} is concerned, note that 
\begin{align*} 
\notag
\sum_{i <j} &\abs{X_{i,j}^N(\eta^N ) - 
	\condexpect{X_{i,j}^N (\eta^N )}{\xi^N}}
\lb \xi^N_j  + \zeta^N_j \rb 
\\
\notag
&\leq 
\sum_{i <j} \abs{X_{i,j}^N - 
		\condexpect{X_{i,j}^N }{\xi^N}}
	\lb 2\xi^N_j  + \abs{\xi^N_j - \zeta^N_j} \rb 
	\\
	&
	\leq 2\sum_{i <j}  \abs{X_{i,j}^N - 
			\condexpect{X_{i,j}^N }{\xi^N}} 
		\sqrt{\xi^N_i \xi^N_j} + 2 d (\xi^N , \zeta^N).
\end{align*} 
Hence,  recalling that $\sum_i \xi^N_i =1$, we infer by Cauchy-Schwarz and \eqref{eq:tX-var} that   
\begin{equation*}
\E \lb 
\sum_{i <j}  \abs{X_{i,j}^N (t ) - 
	\condexpect{X_{i,j}^N (t )}{\xi^N (t )}} 
\sqrt{\xi^N_i (t ) \xi^N_j (t )}\rb \leq \sqrt{\frac{C^\prime}{N}} .
\end{equation*}
Putting these bounds together we conclude that 
the $\E$-expectation of the expression in \eqref{Case3} is bounded above as 
\begin{align} 
\label{eq:Case3-bound} 
&\E\lb 
\sum_{i <j} \lb \abs{X_{i,j}^N(t ) - 
	\condexpect{X_{i,j}^N (t )}{\xi^N (t )}} 
+ \abs{U_{i,j}^N(\xi^N (t))
	- U_{i,j}^N(\zeta^N (t))} \rb
\lb \xi^N_j (t ) + \zeta^N_j (t )\rb
\rb \\
\notag 
&\quad \leq 
\frac{2\sqrt{C^\prime}}{\sqrt{N}} + 15 \E 
\lb  d (\xi^N (t ) , \zeta^N (t ))\rb  \leq 
\frac{2\sqrt{C^\prime}}{\sqrt{N}} + 15 \E \lb \widetilde{\delta}^N (t )\rb .
\end{align}
{\bf Upper bound on the $\E$-expectation of \eqref{Case4}.} 
By the second of \eqref{eq:UV-estimates}, 
\begin{align*}
\sum_{j , k}\abs{Z_{j , k}(\xi^N (t))
	- Z_{j, k}(\zeta^N (t))} 
\frac{\xi^N_j (t ) + \zeta^N (t )}{2} &\leq 
\sum_{j , k} \abs{Z_{j , k}(\xi^N (t))
	- Z_{j, k}(\zeta^N (t))} \\ 
\notag 
& \leq 7 d (\xi^N (t ) , \zeta^N (t )) 
\end{align*}
On the other hand in view of 
\eqref{condexpect Z bound}, 
\begin{align*} 
&\E \lb \sum_{j , k} \abs{Z_{j,k}^N (t ) - 
	\condexpect{Z_{j , k}^N(t )}{\xi^N (t )}} 
\frac{\xi^N_j (t ) + \zeta^N_j (t )}{2}
 \rb
\\ 
&\leq 
\notag 
\E\lb \frac{C}{N^{1/4}} \sum_j \frac{\xi^N_j (t ) \lb 
	\xi^N_j (t ) + \zeta^N_j (t )\rb}{2}\rb 
\leq \frac{C}{N^{1/4}}. 
\end{align*}
Putting these bounds together we conclude that 
the $\E$-expectation of the expression in \eqref{Case4} is bounded above as 
\begin{align} 
\label{eq:Case4-bound} 
&\E\lb 
\sum_{j , k} \lb \abs{Z_{j,k}^N (t ) - 
	\condexpect{Z_{j , k}^N(t )}{\xi^N (t )}} 
+ \abs{Z_{j , k}(\xi^N (t))
	- Z_{j, k}(\zeta^N (t))} \rb
\frac{\xi^N_j (t ) +\zeta_j^N (t )}{2} \rb \\ 
\notag
&
\qquad \qquad 
\leq
\frac{C}{N^{1/4}} + 
15 \E \lb \widetilde{\delta}^N (t )\rb .
\end{align} 
{\bf Proof of Proposition~\ref{prop:Gronwall}.} 
Readily follows from Lemma~\ref{lem:d-growth} and upper bounds
\eqref{eq:Case3-bound}, \eqref{eq:Case4-bound}.\qed

\appendix
\section{Exact formulas for conditional covariances}
\label{app:A}
The computations behind the formulas listed below are based 
on the fact that under the uniform measure $\mu^N$ on $\Sigma^N$ the conditional on cycle structure $\xi^N$  distribution of vertices into particular cycles is 
multinomial 
$\mathsf{Multi}\lb N ; \xi^N_1 , \xi^N_2 , \dots\rb$.
We will use the following Union Jack partition of the set of pairs of integers $U_m:=\{1,\dots,m-1\}\times\{1,\dots,m-1\}$
\begin{align*}
&
C_m:=\{(k,l)\in U_m : k=l=m/2\}, 
\\
&
A_m:=\{(k,l)\in U_m: k=l\not=m/2\}\cup\{(k,l)\in U_m: k=m-l\not=m/2\}, 
\\
&
G_m:=\{(k,l)\in U_m: k\not=l=m/2\}\cup\{(k,l)\in U_m: l\not=k=m/2\},
\\
&
R_m:=U_m\setminus(C_m\cup A_m \cup G_m).
\end{align*}
The symbols $U$, $C$, $A$, $G$ and $R$ denote in turn \emph{Union Jack}, \emph{Centre}, \emph{St Andrew's Cross}, \emph{St George's Cross},  and \emph{The Rest}. 

\begin{lemma}
	\label{lem:cond-rates}
	\begin{align}
	\label{cond-phi-squared}
	&
	\condexpect{\varphi^{N}_{i,j,\mathsf{b}}\varphi^{N}_{i,j,\mathsf{c}}}{\xi^N}
	=
	\\
	\notag
	&
	\qquad
	=
	\begin{cases} 
	\begin{aligned}
	\frac{2N}{N-1} 
	\xi^N_i \xi^N_j
	\end{aligned}
	&\text{if } 
	\abs{\mathsf{b}\cap\mathsf{c}}=2,
	\\[10pt]
	\begin{aligned}
	\frac{N^2}{(N-1)(N-2)}
	\xi^N_i \xi^N_j\left(\xi^N_i +\xi^N_j -\frac{2}{N}\right)  
	\end{aligned}
	&\text{if } 
	\abs{\mathsf{b}\cap\mathsf{c}}=1,
	\\[10pt]
	\begin{aligned}
	\frac{4N^3}{(N-1)(N-2) (N-3)}
	\xi^N_i \xi^N_j\left(\xi^N_i-\frac{1}{N}\right)\left(\xi^N_j-\frac{1}{N}\right)
	\end{aligned}
	&\text{if } 
	\abs{\mathsf{b}\cap\mathsf{c}}=0.
	\end{cases}
	\end{align}
	\begin{align}
	\label{cond-psi-squared}
	&
	\condexpect{\psi^{N}_{i,l,\mathsf{b}}\psi^{N}_{i,l^{\prime},\mathsf{c}}}{\xi^N}
	=
	\one_{\{ N\xi^{N}_i\geq \max\{l,l^{\prime}\}+1\}}
	\times
	\\
	\notag
	&
	\qquad
	\times
	\begin{cases} 
	\begin{aligned}
	&
	\frac{\xi^N_i}{N-1} 
	\big(
	\one_{\{(l,l^\prime)\in C_{N\xi^{N}_i}\}}
	+
	\frac12
	\one_{\{(l,l^\prime)\in A_{N\xi^{N}_i}\}}
	\big)
	\end{aligned}
	&\text{if } 
	\abs{\mathsf{b}\cap\mathsf{c}}=2,
	\\[20pt]
	\begin{aligned}
	&
	\frac{\xi^N_i }{(N-1)(N-2)}
	\big(
	\frac12
	\one_{\{(l,l^\prime)\in A_{N\xi^{N}_i}\}}
	+
	\one_{\{(l,l^\prime)\in G_{N\xi^{N}_i}\}}
	+
	\one_{\{(l,l^\prime)\in R_{N\xi^{N}_i}\}}
	\big)
	\end{aligned}
	&\text{if }
	\abs{\mathsf{b}\cap\mathsf{c}}=1, 
	\\[20pt]
	\begin{aligned}
	&
	\frac{N(\xi^N_i)^2}{(N-1)(N-2) (N-3)}
	-
	\frac{\xi^N_i}{(N-1)(N-2) (N-3)}\times
	\\[10pt]
	&
	\big(
	2
	\one_{\{(l,l^\prime)\in C_{N\xi^{N}_i}\}}
	+
	3
	\one_{\{(l,l^\prime)\in A_{N\xi^{N}_i}\}}
	+
	4
	\one_{\{(l,l^\prime)\in G_{N\xi^{N}_i}\}}
	+
	4
	\one_{\{(l,l^\prime)\in R_{N\xi^{N}_i}\}}
	\big)
	\end{aligned}
	&\text{if } 
	\abs{\mathsf{b}\cap\mathsf{c}}=0. 
	\end{cases}
	\end{align}
\end{lemma} 
{
\begin{proof}
	[Proof of Lemma \ref{lem:cond-rates}]
	The proof of these identities is elementary - though, tedious - enumerative combinatorics. We omit the details. 
\end{proof}
}

\bigskip

\subsection*{Acknowledgements:}
This work was supported by the Israeli Science Foundation grant 765/18 (for DI), respectively, by EPSRC (UK) Fellowship EP/P003656/1 and NKFI (HU) K-129170 (for BT).

\end{document}